\documentclass[11pt,letterpaper,reqno]{amsart}

\title[Stabilization of Euler--Poincar\'e Systems with Broken Symmetry~II]{Controlled Lagrangians and Stabilization of Euler--Poincar\'e Mechanical Systems\\with Broken Symmetry II: Potential Shaping} 
\author{C\'esar Contreras}
\author{Tomoki Ohsawa}
\address{Department of Mathematical Sciences, The University of Texas at Dallas, 800 W Campbell Rd, Richardson, TX 75080-3021}
\email{cxc145430@utdallas.edu,tomoki@utdallas.edu}
\date{\today}



\usepackage{graphicx,amssymb,subfigure,caption,paralist}
\usepackage[margin=1in, marginpar=.5in]{geometry}

\usepackage[numbers,sort&compress]{natbib}
\usepackage[colorlinks=false]{hyperref}

\usepackage{tikz}
\usetikzlibrary{arrows,snakes,backgrounds,patterns,shapes.geometric,calc}
\newcommand\centerofmass{%
    \tikz[radius=0.4em] {%
        \fill (0,0) -- ++(0.4em,0) arc [start angle=0,end angle=90] -- ++(0,-0.8em) arc [start angle=270, end angle=180];%
        \draw (0,0) circle;%
    }%
}

\usepackage{tikz-cd}

\theoremstyle{plain}
\newtheorem{theorem}{Theorem}
\newtheorem*{theorem*}{Theorem}

\theoremstyle{definition}

\newtheorem{example}[theorem]{Example}

\theoremstyle{remark}
\newtheorem{remark}[theorem]{Remark}

\def\od#1#2{\frac{d#1}{d#2}}
\def\pd#1#2{\frac{\partial #1}{\partial #2}}
\def\fd#1#2{\frac{\delta #1}{\delta #2}}

\def\tpd#1#2{\partial #1/\partial #2}
\def\tfd#1#2{\delta #1/\delta #2}

\def\dzero#1#2{\left.\od{}{#1} #2 \right|_{#1=0}}

\def\parentheses#1{{\left(#1\right)}}
\def\brackets#1{{\left[#1\right]}}

\def\Span{\mathop{\mathrm{span}}\nolimits} 

\def\norm#1{{\left\|#1\right\|}}

\def\R{\mathbb{R}}

\def\defeq{\mathrel{\mathop:}=}

\def\setdef#1#2{{\left\{ #1 \ |\ #2 \right\}}}
\def\ip#1#2{{\left\langle#1,#2\right\rangle}}

\def\eps{\varepsilon}

\def\GL{\mathsf{GL}}
\def\SO{\mathsf{SO}}
\def\SE{\mathsf{SE}}

\def\se{\mathfrak{se}}
\def\so{\mathfrak{so}}

\newenvironment{tbmatrix}{\left[\begin{smallmatrix}}{\end{smallmatrix}\right]}

\def\PB#1#2{\left\{#1,#2\right\}}

\newcommand\ad{\operatorname{ad}}

\makeatletter
\DeclareFontFamily{OMX}{MnSymbolE}{}
\DeclareSymbolFont{MnLargeSymbols}{OMX}{MnSymbolE}{m}{n}
\SetSymbolFont{MnLargeSymbols}{bold}{OMX}{MnSymbolE}{b}{n}
\DeclareFontShape{OMX}{MnSymbolE}{m}{n}{
    <-6>  MnSymbolE5
   <6-7>  MnSymbolE6
   <7-8>  MnSymbolE7
   <8-9>  MnSymbolE8
   <9-10> MnSymbolE9
  <10-12> MnSymbolE10
  <12->   MnSymbolE12
}{}
\DeclareFontShape{OMX}{MnSymbolE}{b}{n}{
    <-6>  MnSymbolE-Bold5
   <6-7>  MnSymbolE-Bold6
   <7-8>  MnSymbolE-Bold7
   <8-9>  MnSymbolE-Bold8
   <9-10> MnSymbolE-Bold9
  <10-12> MnSymbolE-Bold10
  <12->   MnSymbolE-Bold12
}{}

\let\llangle\@undefined
\let\rrangle\@undefined
\DeclareMathDelimiter{\llangle}{\mathopen}%
                     {MnLargeSymbols}{'164}{MnLargeSymbols}{'164}
\DeclareMathDelimiter{\rrangle}{\mathclose}%
                     {MnLargeSymbols}{'171}{MnLargeSymbols}{'171}
\makeatother
\def\metric#1#2{\llangle #1, #2\rrangle}

\def\bOmega{\boldsymbol{\Omega}}
\def\bPi{\boldsymbol{\Pi}}
\def\bP{\mathbf{P}}
\def\bGamma{\boldsymbol{\Gamma}}
\def\bTheta{\boldsymbol{\Theta}}
\def\bDelta{\boldsymbol{\Delta}}
\def\vd{\mathbf{v}_{\rm d}}

\begin{document}

\footskip=.5in

\begin{abstract}
  We apply the method of controlled Lagrangians by potential shaping to Euler--Poincar\'e mechanical systems with broken symmetry.
  We assume that the configuration space is a general semidirect product Lie group $\mathsf{G} \ltimes V$ with a particular interest in those systems whose configuration space is the special Euclidean group $\mathsf{SE}(3) = \mathsf{SO}(3) \ltimes \mathbb{R}^{3}$.
  The key idea behind the work is the use of representations of $\mathsf{G} \ltimes V$ and their associated advected parameters.
  Specifically, we derive matching conditions for the modified potential exploiting the representations and advected parameters.
  Our motivating examples are a heavy top spinning on a movable base and an underwater vehicle with non-coincident centers of gravity and buoyancy.
  We consider a few different control problems for these systems, and show that our results give a general framework that reproduces our previous work on the former example and also those of Leonard on the latter.
  Also, in one of the latter cases, we demonstrate the advantage of our representation-based approach by giving a simpler and more succinct formulation of the problem.
\end{abstract}

\maketitle

\section{Introduction}
\subsection{Motivating Example}
The main goal of this paper is to stabilize equilibria of those mechanical systems whose configuration space is a semidirect product Lie group, but whose symmetry is broken by an external force.
While our main results apply to a class of mechanical systems in any finite-dimensional semidirect product Lie group $\mathsf{S} = \mathsf{G} \ltimes V$ with a Lie group $\mathsf{G}$ and a vector space $V$, our main source of motivation is those systems that are naturally defined on the special Euclidean group $\SE(3) \defeq \SO(3) \ltimes \R^{3}$ but do not possess the full $\SE(3)$-symmetry.

Although $\SE(3)$ is the natural configuration of rigid body dynamics, one rarely uses the group explicitly in its formulation, because one can usually decouple the dynamics into the translational one of the center of mass and the rotational one about it.
Furthermore, the rotational dynamics possesses the $\SO(3)$-symmetry because the gravity does not affect it.

This is not the case with the systems shown in Fig.~\ref{fig:motivating_examples}.
For the underwater vehicle (see, e.g., \citet{Le1997a,Le1997b,LeMa1997,WoLe2002} and \citet{ChHaSmWi2007,SmChWiCa2009}), the rotational and translational dynamics are coupled due to the interactions between the vehicle and the surrounding water.
The heavy top rotating on a movable base (which is assumed to be a point mass for simplicity) from our previous work~\cite{CoOh-EPwithBSym1} is essentially the same:
One needs to take into account interactions between the rotational dynamics of the top and the translational dynamics of the base.
Therefore, one needs to formulate both systems on $\SE(3)$.

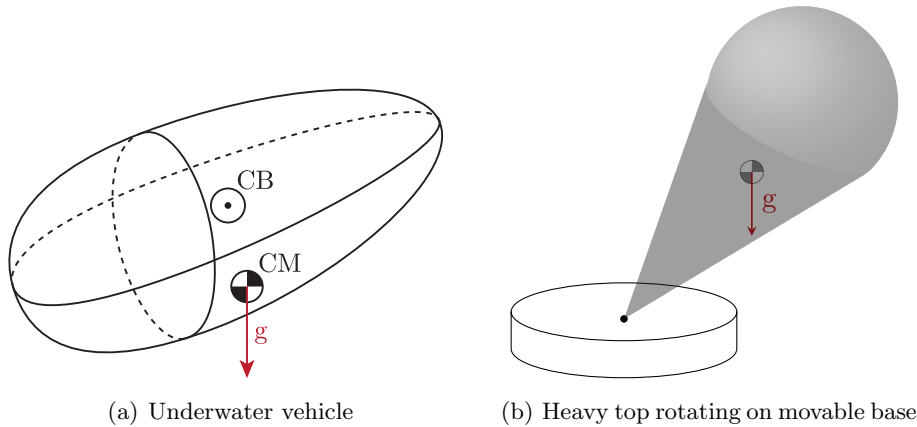
\begin{figure}[hbtp]
  \centering
  \subfigure[Underwater vehicle]{
    \includegraphics[width=0.35\linewidth]{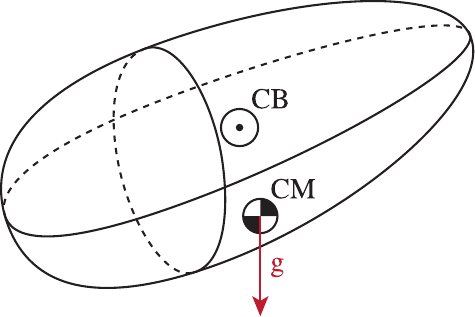}
    \label{fig:uwv}
  }
  \quad
  \subfigure[Heavy top rotating on movable base]{
    \begin{tikzpicture}[scale=.75]
      \node[draw, cylinder, shape aspect=3, rotate=90, minimum height=1cm, minimum width=3cm] (c1)  at (0,1){};
      \draw[fill] (0,1.5) circle [radius=0.05];

      \node[semitransparent] at (2,3.8) {\centerofmass};
      \draw[semithick,red!75!black,->,>=stealth]  (2,3.8) -- (2,2.8);
      \node[right] at (2,3.3) {\color{red!75!black}{$\mathrm{g}$}};
      
      \fill[
      left color=gray!50!black,
      middle color=gray!50,
      right color=gray!50!black,
      shading=axis,
      shading angle=35,
      opacity=0.25
      ] (3.75,3.75) -- (0,1.5) -- (1.3,5.2) [rotate=335] arc (180:340:1.49cm and 0.5cm); 
      \fill[
      shading=ball,
      ball color=gray!50,
      opacity=0.25
      ] (1.3,5.2) [rotate=335] arc (180:340:1.49cm and 0.5cm) arc (340:560:1.507cm); 
    \end{tikzpicture}
    \label{fig:htmb}
  }
  \caption{
    (a)~Underwater vehicle:
    The configuration space is the semidirect product $\SE(3) \defeq \SO(3) \ltimes \R^{3}$, i.e., rotations around the center of buoyancy (CB) and its translational positions.
    The center of mass (CM) is not coincident with the CB; this breaks the $\SE(3)$-symmetry that the system would otherwise possess.
    (b)~Heavy top on a (point-mass) movable base:
    Just like the underwater vehicle, the configuration space is $\SE(3)$, rotations around the junction point and the translational positions of the base; the gravity breaks the $\SE(3)$-symmetry.
  }
  \label{fig:motivating_examples}
\end{figure}

Moreover, the gravity breaks the symmetry of both systems.
The underwater vehicle is subject to both buoyancy and gravity, which usually act on different centers of the body.
One is therefore compelled to select either of them---say the center of buoyancy here---as the center of rotation; then the gravity breaks the $\SE(3)$-symmetry.
For the heavy top on a movable base, the natural center of the translational and rotational motions would be the junction point of the top and the base, but the center of mass of the top is not at the junction point, thereby breaking the $\SE(3)$-symmetry as well.

The broken symmetry implies that the standard Euler--Poincar\'e or Lie--Poisson theory does not directly apply to these systems.
To remedy the broken $\mathsf{S}$-symmetry, one needs to introduce advected parameters via a representation of $\mathsf{S}$ on the dual $X^{*}$ of an appropriate vector space $X$.
From the Lagrangian point of view, this results in the Euler--Poincar\'e equation with advected parameters on $\mathfrak{s} \times X^{*}$; see~\citet{HoMaRa1998a} and \citet{CeHoMaRa1998}.

The advantages of the Euler--Poincar\'e equation with advected parameters are: (i)~the resulting equations are defined on a vector space $\mathfrak{s} \times X^{*}$ as opposed to the tangent bundle $T\mathsf{S}$ of the Lie group $\mathsf{S}$; (ii)~the reduced Lagrangian defined on $\mathfrak{s} \times X^{*}$ tends to have a simpler expression than the original one defined on $T\mathsf{S}$.
As a result, the Euler--Poincar\'e equation with advected parameters are amenable to the method of controlled Lagrangians~\cite{Ha1999,Ha2000,BlChLeMa2001,BlLeMa2000,BlLeMa2001,ChBlLeMaWo2002,ChMa2004,BlOrVa2002,OrPeNiSi1998,OrScMaMa2001}, because a simpler expression of the Lagrangian on a vector space facilitates the derivation of the matching condition.

\subsection{Main Results and Outline}
We apply the method of controlled Lagrangians---using potential shaping particularly---to the Euler--Poincar\'e equation with advected parameters.
This work is a companion paper to our paper~\cite{CoOh-EPwithBSym1} that focused on kinetic shaping of such systems.
Our main results are matching conditions as well as the resulting control laws for such systems using potential shaping for a class of mechanical systems on a semidirect product Lie group $\mathsf{S} = \mathsf{G} \ltimes V$ with broken symmetry.
The key idea is the use of representations of the Lie group $\mathsf{S}$ and their associated advected parameters and momentum maps.
We demonstrate the generality and applicability of the theory by deriving those controls used in some existing works.

We note that the matching condition we seek here is less general than what is usually referred to as matching conditions (see, e.g., \citet{BlOrVa2002}) in which one obtains a PDE for the controlled Lagrangian.
Our matching conditions are simplified due to a specific form of potential shaping ansatz, and also do \textit{not} systematically characterize the stability of the system.
Instead, our matching conditions provide the first step towards stability: The matching must be followed by an analysis of stability conditions for each specific system in order to find an explicit stabilizing control law.
It would be an interesting future work to generalize our approach to encompass stabilization without assuming a specific ansatz for the controlled Lagrangian.

The idea of potential shaping has been around for quite a while and has been studied quite extensively in various settings; see, e.g., \citet{Sc1986}, \citet{NiSc1990}, \cite{Ha1999}, \citet{OrSpGoBl2002,BlOrVa2002,OrPeNiSi1998,OrScMaMa2001}, \citet{BlLeMa1999b,BlChLeMa2001}, \citet[Section~10.4]{BuLe2004}, \citet{SpBu2005}, and \citet{WoTe2009}.
However, none of those works address matching conditions for the Euler--Poincar\'e equation with advected parameters in general, nor stresses the role of Lie group representations.

The paper proceeds as follows:
We first give a brief survey of semidirect product Lie groups in Section~\ref{sec:semidirect_product} in order to make the paper self-contained as well as to set the notation straight, because notations involving various representations used in the semidirect product theory can be quite confusing.

In Section~\ref{sec:mechanical_systems}, we build on Section~\ref{sec:semidirect_product} to formulate the basic equations of mechanical systems on semidirect product Lie groups with broken symmetry---the Euler--Poincar\'e equation with advected parameters.
We then work out the examples shown in Fig.~\ref{fig:motivating_examples} to illustrate the ideas.
We also show how to track additional advected parameters.
This idea is important in stabilizing an equilibrium that is characterized by additional variables than the original variables of the system.

In Section~\ref{sec:potential_shaping}, we consider controlled Euler--Poincar\'e equation with advected parameters with potential shaping, and derive matching conditions as well as the resulting control laws.
Particularly, we consider the following two settings:
\begin{enumerate}[(i)]
\item The controlled system becomes a simpler system with \textit{less} advected parameters.
  This boils down to considering a subrepresentation of the original representation used to describe the original advected parameters.
  \smallskip
\item The controlled system involves \textit{additional} advected parameters---hence additional representations.
  Specifically, an operational goal of the system naturally gives rise to an equilibrium defined in terms of the original configuration variables and additional advected parameters.
\end{enumerate}
So in both cases, it boils down to using proper representations.
As a result, the matching conditions we derive are in terms of those momentum maps associated with these representations.

The first setting is rather restrictive because one can manage to reduce advected parameters in limited circumstances.
On the other hand, the second setting would have more applications because one has much more freedom in introducing additional advected parameters than reducing them, oftentimes for practical purposes.

As an example of the first setting, we find the ad-hoc potential shaping applied to the system in Fig.~\ref{fig:htmb} from our work~\cite{CoOh-EPwithBSym1}.
For the second one, we obtain those controls found by \citet{Le1997b} to stabilize a desired steady motion as well as to prevent translational drift in underwater vehicles.
Particularly, in finding the control to prevent translational drift, our use of representation of $\SE(3)$ on $\R^{4} \times \R^{4}$ results in a simpler formulation of the problem than in \citet{Le1997b}, thereby demonstrating the efficacy of our approach.

\section{Semidirect Product Lie Groups}
\label{sec:semidirect_product}
Although the concept of semidirect product Lie groups is fairly well known, derivations of concrete formulas in such Lie groups can be quite involved, and are usually not covered with details in standard references.
So we give a short survey of semidirect product Lie groups using $\SE(3) \defeq \SO(3) \ltimes \R^{3}$ as a running example to illustrate concrete calculations.
Our main references here are \citet{MaRaWe1984a,MaRaWe1984b,HoMaRa1998a,CeHoMaRa1998}.
This section overlaps with the companion paper \cite{CoOh-EPwithBSym1}, but is included for completeness as well as to set the notation.

\subsection{Semidirect Product Lie Groups and Lie Algebras}
Let $\mathsf{G}$ be a Lie group, $V$ be a vector space, and $\GL(V)$ be the set of all invertible linear transformations on $V$.
Let $\rho\colon \mathsf{G} \to \GL(V)$ be a (left) representation of $\mathsf{G}$ on $V$, i.e., $\rho(g_{1} g_{2}) = \rho(g_{1}) \rho(g_{2})$ for any $g_{1}, g_{2} \in \mathsf{G}$.
We define the semidirect product Lie group $\mathsf{S} \defeq \mathsf{G} \ltimes V$ under the multiplication
\begin{align*}
  s_{1} \cdot s_{2}
  = (g_{1}, x_{1}) \cdot (g_{2}, x_{2})
  = (g_{1} g_{2}, \rho(g_{1}) x_{2} + x_{1}).
\end{align*}
Therefore, for any element $s = (g, x) \in \mathsf{S}$, its inverse is defined by
\begin{equation*}
  s^{-1} = (g,x)^{-1} = \parentheses{ g^{-1}, -\rho(g^{-1})x }.
\end{equation*}

\begin{example}[$\SE(3) = \SO(3) \ltimes \R^{3}$]
  Consider the representation
  \begin{equation*}
    \rho\colon \SO(3) \to \GL(\R^{3}) = \GL(3,\R);
    \qquad
    \rho(R)\mathbf{x} = R \mathbf{x}
  \end{equation*}
  defined by the standard matrix-vector multiplication.
  Then we can define the special Euclidean group $\SE(3) \defeq \SO(3) \ltimes \R^{3}$ under the following group multiplication:
  \begin{equation*}
    (R_{1}, \mathbf{x}_{1}) \cdot (R_{2}, \mathbf{x}_{2})
    = (R_{1} R_{2}, R_{1} \mathbf{x}_{2} + \mathbf{x}_{1}).
  \end{equation*}
  One may think of $(R_{2}, \mathbf{x}_{2})$ as the rotational and translational configurations of a rigid body in $\R^{3}$, and then may see the above operation as the rotation by $R_{1}$ followed by the translation by $\mathbf{x}_{1}$ applied to the old configuration $(R_{2}, \mathbf{x}_{2})$.
  Another way of looking at $\SE(3)$ is that it is the matrix group
  \begin{equation*}
    \SE(3) = \setdef{ (R,\mathbf{x}) \defeq \begin{bmatrix} R & \mathbf{x} \\ \mathbf{0}^{T} & 1 \end{bmatrix} }{ R \in \SO(3),\, \mathbf{x} \in \mathbb{R}^{3} }
  \end{equation*}
  under the standard matrix multiplication.
\end{example}

\subsection{Induced Representations}
The representation $\rho$ induces several other representations as well.
First, the dual $\rho^{*}\colon \mathsf{G} \to \mathsf{GL}(V^{*})$ is defined so that, for any $g \in \mathsf{G}$, any $\alpha \in V^{*}$, and any $x \in V$,
\begin{equation*}
  \ip{ \rho^{*}(g) \alpha }{ x }
  = \ip{ \alpha }{ \rho(g^{-1}) x },
\end{equation*}
where $\ip{\,\cdot\,}{\,\cdot\,}\colon V^{*} \times V \to \R$ is the natural dual pairing.
This yields $\rho^{*}(g) = \rho(g^{-1})^{*}$, and indeed defines a left representation of $\mathsf{G}$ on $V^{*}$.

Let $\mathfrak{g}$ be the Lie algebra of $\mathsf{G}$.
Then the Lie group representation $\rho$ also gives rise to the Lie algebra representation $\rho'\colon \mathfrak{g} \to \mathfrak{gl}(V)$ as follows:
\begin{equation*}
  \rho'(\xi) v \defeq \dzero{\eps}{ \rho(\exp(\eps\xi)) v } = \xi_{V}(v),
\end{equation*}
where $\xi_{V}$ is the infinitesimal generator on $V$ corresponding to $\xi$.
In fact, as shown in \cite[Proposition~9.1.6]{MaRa1999}, $\rho'$ is a Lie algebra homomorphism, i.e., for any $\xi, \eta \in \mathfrak{g}$,
\begin{equation*}
  \rho'([\xi,\eta]) = [\rho'(\xi), \rho'(\eta)].
\end{equation*}

The Lie algebra $\mathfrak{s}$ of $\mathsf{S}$ is the semidirect product Lie algebra $\mathfrak{g} \ltimes V$ equipped with the following commutator or adjoint operator:
\begin{equation*}
  \ad_{(\xi,v)}(\eta,w)
  \defeq [(\xi,v), (\eta,w)]
  = \parentheses{
    \ad_{\xi}\eta,\,
    \rho'(\xi)w - \rho'(\eta)v
  }.
\end{equation*}

Let us next find the coadjoint representation on the dual $\mathfrak{s}^{*}$ of the Lie algebra $\mathfrak{s}$.
To that end, we first would like to find the so-called diamond operator (see \citet{HoMaRa1998a,CeHoMaRa1998} and \citet[\S7.5]{HoScSt2009}).
Let us fix $v \in V$ in $\rho'(\xi) v$ to regard $\xi \mapsto \rho'(\xi) v$ as a linear map $\rho'_{v}\colon \mathfrak{g} \to V$, i.e.,
\begin{equation*}
  \rho'_{v}(\xi) \defeq \rho'(\xi) v.
\end{equation*}
Then its dual map $(\rho'_{v})^{*}\colon V^{*} \to \mathfrak{g}^{*}$ is defined so that, for any $\alpha \in V^{*}$ and $\xi \in \mathfrak{g}$,
\begin{equation*}
  \ip{ (\rho'_{v})^{*}(\alpha) }{ \xi } = \ip{ \alpha }{ \rho'_{v}(\xi) }.
\end{equation*}
The diamond operator $\diamond\colon V \times V^{*} \to \mathfrak{g}^{*}$ is then defined as
\begin{equation}
  \label{eq:diamond}
  v \diamond \alpha \defeq (\rho'_{v})^{*}\alpha.
\end{equation}
The diamond operator is actually the momentum map associated with the cotangent lift of the action defined by the representation $\rho$.
In fact, for any $\alpha \in V^{*}$ and any $\xi \in \mathfrak{g}$,
\begin{align*}
  \ip{ (\rho'_{v})^{*}\alpha }{ \xi }
  &= \ip{ \alpha }{ \rho'_{v}(\xi) } \\
  &= \ip{ \alpha }{ \rho'(\xi)v } \\
  &= \ip{ \alpha }{ \dzero{\eps}{\rho(\exp(s\eps))v} } \\
  &= \ip{ \alpha }{ \xi_{V}(v) } \\
  &= \ip{ \mathbf{J}(v,\alpha) }{ \xi },
\end{align*}
where $\mathbf{J}\colon T^{*}V \cong V \times V^{*} \to \mathfrak{g}^{*}$ is the momentum map associated with the cotangent lift of the $\mathsf{G}$-action $\rho$ on $V$.
Therefore,
\begin{equation}
  \label{eq:diamond2}
  v \diamond \alpha \defeq (\rho'_{v})^{*}\alpha = \mathbf{J}(v,\alpha).
\end{equation}

Using the diamond operator or the momentum map $\mathbf{J}$, we may write the coadjoint representation on the dual $\mathfrak{s}^{*}$ of $\mathsf{S}$ as follows:
\begin{equation}
  \label{eq:adstar}
  \ad_{(\xi,v)}^{*}(\mu,\alpha)
  = \parentheses{
    \ad_{\xi}^{*}\mu - \mathbf{J}(v,\alpha),\,
    \rho'(\xi)^{*}\alpha
  },
\end{equation}
where  $\rho'(\xi)^{*}$ is the dual map of $\rho'(\xi)$, i.e.,
\begin{equation*}
  \ip{ \rho'(\xi)^{*} \alpha }{ v } = \ip{ \alpha }{ \rho'(\xi) v }.
\end{equation*}

\begin{example}[$\SE(3) = \SO(3) \ltimes \R^{3}$]
  Identifying $(\R^{3})^{*}$ with $\R^{3}$ via the dot product, the dual $\rho^{*}$ is defined as
  \begin{equation*}
    \parentheses{ \rho^{*}(R) \boldsymbol{\alpha} } \cdot \mathbf{x}
    = \boldsymbol{\alpha} \cdot \rho(R^{-1}) \mathbf{x}
    = \boldsymbol{\alpha} \cdot (R^{-1} \mathbf{x})
    = (R \boldsymbol{\alpha}) \cdot \mathbf{x},
  \end{equation*}
  and so $\rho^{*}(R) \boldsymbol{\alpha} = R \boldsymbol{\alpha}$.
  
  Let us introduce the hat map to identify $\so(3)$ with $\R^{3}$:
  \begin{equation*}
    \hat{(\,\cdot\,)}\colon \R^{3} \to \so(3);
    \qquad
    \mathbf{a} \mapsto
    \hat{a} \defeq
    \begin{bmatrix}
      0 & -a_{3} & a_{2} \\
      a_{3} & 0 & -a_{1} \\
      -a_{2} & a_{1} & 0
    \end{bmatrix}.
  \end{equation*}
  Then we have $\hat{a} \mathbf{b} = \mathbf{a} \times \mathbf{b}$, and $[\hat{a}, \hat{b}]$ is identified with $\mathbf{a} \times \mathbf{b}$.
  The Lie algebra representation $\rho'$ is then
  \begin{equation}
    \label{eq:rho'-se3}
    \rho'(\hat{\Omega}) \mathbf{v} = \rho_{\mathbf{v}}'(\hat{\Omega}) = \dzero{\eps}{ \exp( \eps\hat{\Omega} ) \mathbf{v} }
    = \hat{\Omega} \mathbf{v}
    = \bOmega \times \mathbf{v}.
  \end{equation}
  As a result, we can express the commutator as
  \begin{equation*}
    \ad_{(\hat{\Omega}, \mathbf{v})} (\hat{\eta}, \mathbf{w})
    = [(\hat{\Omega}, \mathbf{v}), (\hat{\eta}, \mathbf{w})]
    = \parentheses{
      [\hat{\Omega}, \hat{\eta}],\,
      \hat{\Omega} \mathbf{w} - \hat{\eta} \mathbf{v}
    }
  \end{equation*}
  or in terms of vectors in $\R^{3}$,
  \begin{equation*}
    \ad_{(\bOmega, \mathbf{v})} (\boldsymbol{\eta}, \mathbf{w})
    = [(\bOmega, \mathbf{v}), (\boldsymbol{\eta}, \mathbf{w})]
    = \parentheses{
      \bOmega \times \boldsymbol{\eta},\,
      \bOmega \times \mathbf{w} - \boldsymbol{\eta} \times \mathbf{v}
    }.
  \end{equation*}
  Let us find the diamond operator.
  We have, for any $\hat{\Omega} \in \so(3)$,
  \begin{equation*}
    \ip{ (\rho_{\mathbf{v}}')^{*}\alpha }{ \hat{\Omega} }
    =  (\rho_{\mathbf{v}}')^{*}(\boldsymbol{\alpha}) \cdot \bOmega
    = \boldsymbol{\alpha} \cdot \parentheses{ \rho_{\mathbf{v}}'(\hat{\Omega}) }
    = \boldsymbol{\alpha} \cdot \parentheses{ \bOmega \times \mathbf{v} }
    = \parentheses{ \mathbf{v} \times \boldsymbol{\alpha} } \cdot \bOmega,
  \end{equation*}
  and so
  \begin{equation*}
    \mathbf{v} \diamond \boldsymbol{\alpha}
    = (\rho_{\mathbf{v}}')^{*}(\boldsymbol{\alpha})
    = \mathbf{v} \times \boldsymbol{\alpha}.
  \end{equation*}
  We may also find the dual $\rho'(\hat{\Omega})^{*}$ as follows:
  \begin{equation*}
    \ip{ \rho'(\hat{\Omega})^{*}\boldsymbol{\alpha} }{ \mathbf{v} }
    = \ip{ \boldsymbol{\alpha} }{ \rho'(\hat{\Omega})\mathbf{v} }
    = \boldsymbol{\alpha} \cdot (\bOmega \times \mathbf{v})
    = ( \boldsymbol{\alpha} \times \bOmega ) \cdot \mathbf{v},
  \end{equation*}
  and so
  \begin{equation*}
    \rho'(\hat{\Omega})^{*}\boldsymbol{\alpha} = \boldsymbol{\alpha} \times \bOmega.
  \end{equation*}
  As a result, we may write the coadjoint action as follows:
  \begin{equation*}
    \ad_{(\bOmega, \mathbf{v})}^{*} (\boldsymbol{\mu}, \boldsymbol{\alpha})
    = \parentheses{
       \boldsymbol{\mu} \times \bOmega - \mathbf{v} \times \boldsymbol{\alpha},\,
       \boldsymbol{\alpha} \times \bOmega
    }.
  \end{equation*}
\end{example}

\section{Mechanical Systems on Semidirect Product Lie Groups with Broken Symmetry}
\label{sec:mechanical_systems}
\subsection{Broken Symmetry}
Let $\mathsf{S} \defeq \mathsf{G} \ltimes V$ be a semidirect product Lie group, and $L_{a_{0}}\colon T\mathsf{S} \to \R$ be a Lagrangian with parameters $a_{0} \in X^{*}$, where $X^{*}$ is the dual of a vector space $X$.
Specifically, we assume that the Lagrangian takes the following form:
\begin{equation*}
  L_{a_{0}}(s, \dot{s}) = \frac{1}{2} \metric{\dot{s}}{\dot{s}} - U_{a_{0}}(s),
\end{equation*}
where $\metric{\,\cdot\,}{\,\cdot\,}$ is a left-invariant metric on $T\mathsf{S}$, i.e., for any $s, s_{0} \in \mathsf{S}$ and any $\dot{s} \in T_{s}\mathsf{S}$,
\begin{equation*}
  \metric{T_{s}\mathsf{L}_{s_{0}}(\dot{s})}{T_{s}\mathsf{L}_{s_{0}}(\dot{s})} = \metric{\dot{s}}{\dot{s}},
\end{equation*}
where $\mathsf{L}$ stands for the left translation, i.e., $\mathsf{L}_{s_{0}}(s) = s_{0} s$ for any $s_{0}, s \in \mathsf{S}$, and $T\mathsf{L}$ is its tangent lift.
On the other hand, the potential is \textit{not} $\mathsf{S}$-invariant, i.e., $U_{a_{0}}(s_{0} s) \neq U_{a_{0}}(s)$ for some $s_{0}, s \in \mathsf{S}$, and thus breaks the $\mathsf{S}$-symmetry.

\subsection{Recovery of Symmetry}
\label{ssec:recovery}
Suppose that we can recover the broken $\mathsf{S}$-symmetry of the potential in the following way:
Let us first define the extended potential $U\colon \mathsf{S} \times X^{*} \to \R$ so that $U(s, a_{0}) = U_{a_{0}}(s)$ for any $s \in \mathsf{S}$, and let $\sigma\colon \mathsf{S} \to \GL(X)$ be a representation of $\mathsf{S}$ on $X$, and $\sigma^{*}\colon \mathsf{S} \to \GL(X^{*})$ be the induced representation on the dual $X^{*}$.
We assume that we can find an appropriate $\sigma$ so that we can recover the $\mathsf{S}$-symmetry of the potential, i.e., for any $s_{0}, s \in \mathsf{S}$ and any $a \in X^{*}$,
\begin{equation*}
  U\parentheses{ s_{0}s, \sigma(s_{0})^{*}a } = U(s, a).
\end{equation*}

Now let us define the extended Lagrangian $L\colon T\mathsf{S} \times X^{*} \to \R$ by setting
\begin{equation*}
  L(s, \dot{s}, a) \defeq \frac{1}{2} \metric{\dot{s}}{\dot{s}} - U(s, a),
\end{equation*}
and also define the action
\begin{equation*}
  \begin{aligned}
    \Phi\colon &\mathsf{S} \times (T\mathsf{S} \times X^{*}) \to T\mathsf{S} \times X^{*};
    \\
    &(s_{0}, (s, \dot{s}, a)) \mapsto \Phi_{s_{0}}(s, \dot{s}, a) \defeq \parentheses{ s_{0} s, T_{s}\mathsf{L}_{s_{0}}(\dot{s}), \sigma^{*}(s_{0}) a }.
  \end{aligned}
\end{equation*}
Then we see that the extended Lagrangian now possesses the $\mathsf{S}$-symmetry, i.e., $L \circ \Phi_{s_{0}} = L$ for any $s_{0} \in \mathsf{S}$.

\begin{remark}
  It is the variables in the dual space $X^{*}$ that have a practical importance here, whereas the variables in $X$ are auxiliary in nature.
  In the Lagrangian semidirect product theory~\cite{HoMaRa1998a,CeHoMaRa1998}, the significance of having the dual space $X^{*}$ (as opposed to $X$) for the parameters is not particularly clear.
  However, in the Hamiltonian theory~\cite{MaRaWe1984a,MaRaWe1984b}, one can formulate the system as the Lie--Poisson equation on $(\mathfrak{s} \ltimes X)^{*}$, and hence it is rather natural to have the dual space $X^{*}$ here; see \cite{HoMaRa1998a} for a comparison of the Lagrangian and Hamiltonian theories.
\end{remark}

We will also later need the momentum map $\mathbf{K} \colon X \times X^{*} \to \mathfrak{s}^{*}$ associated with the above action $\sigma$.
It is defined analogously to $\mathbf{J}$ from \eqref{eq:diamond} and \eqref{eq:diamond2} as follows:
\begin{equation}
  \label{eq:K}
  \mathbf{K}(x, a)
  = \parentheses{ \mathbf{K}_{\mathfrak{g}^{*}}(x, a),\, \mathbf{K}_{V^{*}}(x, a) }
  \defeq (\sigma'_{x})^{*}a,
\end{equation}
where we split the components of $\mathbf{K}$ into those in $\mathfrak{g}^{*}$ and $V^{*}$ as $\mathbf{K}_{\mathfrak{g}^{*}}$ and $\mathbf{K}_{V^{*}}$.

\subsection{Euler--Poincar\'e Equation with Advected Parameters}
\label{ssec:EPwithAdP}
Once the $\mathsf{S}$-symmetry is recovered as shown above, one may define (with an abuse of notation) the reduced potential $U\colon X^{*} \to \R$ so that $U\parentheses{ \sigma(s^{-1})a } = U\parentheses{ e, \sigma(s^{-1})a }$, i.e.,
\begin{equation*}
  U(a) \defeq U(e, a),
\end{equation*}
and hence also define the reduced Lagrangian $\ell\colon \mathfrak{s} \times X^{*} \to \R$ as
\begin{equation}
  \label{eq:ell}
  \ell(\xi, v, a) \defeq L(e, (\xi, v), a) = \frac{1}{2} \metric{(\xi,v)}{(\xi,v)} - U(a).
\end{equation}
Then one may reduce the variational principle from $T\mathsf{S} \times X^{*}$ to $\mathfrak{s} \times X^{*}$~(see \cite{HoMaRa1998a,CeHoMaRa1998} and \cite[\S7.5]{HoScSt2009}) to obtain the Euler--Poincar\'e equation with advected parameters:
\begin{equation*}
  \od{}{t}\parentheses{ \fd{\ell}{(\xi,v)} } = \ad_{(\xi,v)}^{*} \fd{\ell}{(\xi,v)} + \mathbf{K}\parentheses{ \fd{\ell}{a}, a },
  \qquad
  \dot{a} = \sigma'(\xi,v)^{*} a.
\end{equation*}
Note that, for any smooth function $f\colon E \to \R$ on a real vector space $E$, we define its functional derivative $\tfd{f}{x} \in E^{*}$ at $x \in E$ such that, for any $\delta x \in E$, under the natural dual pairing $\ip{\,\cdot\,}{\,\cdot\,}\colon E^{*} \times E \to \R$,
\begin{equation*}
  \ip{ \fd{f}{x} }{ \delta x } = \dzero{\eps}{ f(x + \eps\delta x) }.
\end{equation*}
For example, if $E = \R^{n}$ and $(\R^{n})^{*}$ is identified with $\R^{n}$ via the dot product, $\tfd{f}{\mathbf{x}}$ is the gradient $\tpd{f}{\mathbf{x}}$.

Using the formula~\eqref{eq:adstar} for the coadjoint action on $\mathfrak{s}^{*}$ as well as the expression for $\mathbf{K}$ in \eqref{eq:K}, we have
\begin{equation}
  \label{eq:EP}
  \begin{split}
    \od{}{t}\parentheses{ \fd{\ell}{\xi} } &= \ad_{\xi}^{*} \fd{\ell}{\xi} - \mathbf{J}\parentheses{v, \fd{\ell}{v}} + \mathbf{K}_{\mathfrak{g}^{*}}\parentheses{ \fd{\ell}{a}, a }, \\
    \od{}{t}\parentheses{ \fd{\ell}{v} } &= \rho'(\xi)^{*} \fd{\ell}{v} + \mathbf{K}_{V^{*}}\parentheses{ \fd{\ell}{a}, a }, \\
    \od{a}{t} &= \sigma'(\xi,v)^{*} a.
  \end{split}
\end{equation}

\begin{example}[{Underwater vehicle; see \citet{Le1997a,Le1997b,LeMa1997}}]
  \label{ex:uwv}
  Consider the underwater vehicle shown in Fig.~\ref{fig:uwv-details}.
  The configuration space is $\mathsf{S} = \SE(3)$, i.e., rotations about the center of buoyancy and its translational positions; see Fig.~\ref{fig:uwv}.
  More specifically, let $\{ \mathbf{e}_{i} \}_{i=1}^{3}$ and $\{ \mathbf{E}_{i} \}_{i=1}^{3}$ be the orthonormal spatial/inertial and body frames, respectively; the body frame is attached to the body at the center of buoyancy (CB) and is taken to be the principal axes of the body; see Fig.~\ref{fig:uwv-details}.
  Then, by defining the matrix $R$ so that $\mathbf{E}_{i} = R \mathbf{e}_{i}$ for $i = 1, 2, 3$ gives an element $R \in \SO(3)$.
  Note that $\{ \mathbf{E}_{i} \}_{i=1}^{3}$ is time-dependent whereas $\{ \mathbf{e}_{i} \}_{i=1}^{3}$ is fixed.
  Moreover, specifying the position of the center of buoyancy in the spatial frame as $\mathbf{x} \in \R^{3}$, we have an element $(R, \mathbf{x})$ in $\SE(3)$ that specifies the orientation and the position of the vehicle.
  
  \begin{figure}[htbp]
    \centering
    \includegraphics[width=0.5\linewidth]{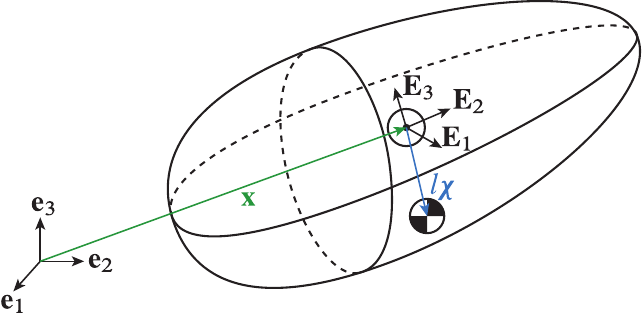}
    \caption{Underwater vehicle}
    \label{fig:uwv-details}
  \end{figure}
   
  The metric $\metric{\,\cdot\,}{\,\cdot\,}$ defining the kinetic energy is left-invariant, and is given as (see \cite{LeMa1997,Le1997a,Le1997b} for details)
  \begin{equation}
    \label{eq:metric-uwv}
    \metric{(\bOmega,\mathbf{v})}{(\bOmega,\mathbf{v})} =
    \begin{bmatrix}
      \bOmega^{T} & \mathbf{v}^{T}
    \end{bmatrix}
    \begin{bmatrix}
      J & D \\
      D^{T} & M
    \end{bmatrix}
    \begin{bmatrix}
      \bOmega \\
      \mathbf{v}
    \end{bmatrix},
  \end{equation}
  where $\bOmega$ and $\mathbf{v}$ are body angular velocity and the velocity of the center of buoyancy seen from the body frame.
  As a result, we may define the the angular and linear impulses~\cite{Le1997a,Le1997b,LeMa1997}:
  \begin{equation}
    \label{eq:impulses}
    \bPi \defeq \fd{\ell}{\bOmega} = J\bOmega + D\mathbf{v},
    \qquad
    \bP \defeq \fd{\ell}{\mathbf{v}} = D^{T}\bOmega + M\mathbf{v}.
  \end{equation}
  
  On the other hand, assuming the neutral buoyancy, the potential term is given as
  \begin{equation*}
    U_{\mathbf{e}_{3}}(R, \mathbf{x})
    = m \mathrm{g} l \mathbf{e}_{3} \cdot (R\boldsymbol{\chi})
    = m \mathrm{g} l \boldsymbol{\chi} \cdot (R^{-1}\mathbf{e}_{3}),
  \end{equation*}
  where $l \boldsymbol{\chi}$ is the position vector---$l$ being its length and $\boldsymbol{\chi}$ being the unit vector for the direction---of the center of mass measured from the center of buoyancy; see Fig.~\ref{fig:uwv-details}.
  Hence we define the extended potential $U\colon \SE(3) \times (\R^{3})^{*} \to \R$ by setting
  \begin{equation*}
    U(R, \mathbf{x}, \bGamma) \defeq m \mathrm{g} l \boldsymbol{\chi} \cdot (R^{-1}\bGamma)
  \end{equation*}
  so that $U(R, \mathbf{x}, \mathbf{e}_{3}) = U_{\mathbf{e}_{3}}(R, \mathbf{x})$.
  
  Also define the representation $\sigma\colon \SE(3) \to \GL(\R^{3})$ by
  \begin{equation*}
    \sigma(R, \mathbf{x}) \mathbf{y} \defeq R \mathbf{y}.
  \end{equation*}
  Identifying $(\R^{3})^{*}$ with $\R^{3}$ via the inner product, we have
  \begin{equation*}
    \parentheses{ \sigma^{*}(R, \mathbf{x}) \bGamma } \cdot \mathbf{y}
    = \bGamma \cdot \parentheses{ \sigma((R, \mathbf{x})^{-1}) \mathbf{y} }
    = \bGamma \cdot \parentheses{ R^{-1} \mathbf{y} }
    = (R \bGamma ) \cdot \mathbf{y}.
  \end{equation*}
  Therefore, we have
  \begin{equation*}
    \sigma^{*}(R, \mathbf{x}) \bGamma = R \bGamma.
  \end{equation*}
  As a result, we have, for any $(R_{0}, \mathbf{x}_{0}), (R, \mathbf{x}) \in \SE(3)$ and any $\bGamma \in \R^{3}$,
  \begin{align*}
    U\parentheses{ (R_{0}, \mathbf{x}_{0}) \cdot (R, \mathbf{x}), \sigma^{*}(R_{0}, \mathbf{x}_{0}) \bGamma }
    &= m \mathrm{g} l \boldsymbol{\chi} \cdot ((R_{0}R)^{-1} R_{0}\bGamma) \\
    &= m \mathrm{g} l \boldsymbol{\chi} \cdot (R^{-1} \bGamma) \\
    &= U(R, \mathbf{x}, \bGamma),
  \end{align*}
  hence recovering the $\SE(3)$-symmetry.
  Then we may define the reduced potential $U\colon (\R^{3})^{*} \to \R$ as
  \begin{equation*}
    U(\bGamma) \defeq U(I,\mathbf{0},\bGamma) = m \mathrm{g} l \boldsymbol{\chi} \cdot \bGamma,
  \end{equation*}
  and the reduced Lagrangian $\ell\colon \se(3) \times (\R^{3})^{*} \to \R$ as
  \begin{equation*}
    \ell(\bOmega, \mathbf{v}, \bGamma)
    \defeq \frac{1}{2}\metric{(\bOmega,\mathbf{v})}{(\bOmega,\mathbf{v})} - m \mathrm{g} l \boldsymbol{\chi} \cdot \bGamma.
  \end{equation*}
  
  We also find
  \begin{equation}
    \label{eq:sigma'-uwv}
    \sigma'(\bOmega,\mathbf{v}) \mathbf{y}
    = \sigma_{\mathbf{y}}'(\bOmega,\mathbf{v})
    = \hat{\Omega} \mathbf{y}
    = \bOmega \times \mathbf{y},
  \end{equation}
  and thus
  \begin{equation*}
    \parentheses{ \sigma'(\bOmega,\mathbf{v})^{*} \bGamma } \cdot \mathbf{y}
    = \bGamma \cdot \parentheses{ \sigma'(\bOmega,\mathbf{v}) \mathbf{y} }
    = \bGamma \cdot (\bOmega \times \mathbf{y})
    = (\bGamma \times \bOmega) \cdot \mathbf{y},
  \end{equation*}
  resulting in
  \begin{equation}
    \label{eq:sigma'star-uwv}
    \sigma'(\bOmega,\mathbf{v})^{*} \bGamma = \bGamma \times \bOmega.
  \end{equation}
  Similarly,
  \begin{equation*}
    \parentheses{ (\sigma'_{\mathbf{y}})^{*}\bGamma } \cdot (\bOmega,\mathbf{v})
    = \bGamma \cdot \sigma'_{\mathbf{y}}(\bOmega,\mathbf{v})
    = \bGamma \cdot (\bOmega \times \mathbf{y})
    = (\mathbf{y} \times \bGamma) \cdot \bOmega,
  \end{equation*}
  and so we obtain, using \eqref{eq:K},
  \begin{equation}
    \label{eq:K-uwv}
    \mathbf{K}(\mathbf{y},\bGamma)
    = \parentheses{ \mathbf{K}_{\so(3)^{*}}(\mathbf{y},\bGamma), \mathbf{K}_{(\R^{3})^{*}}(\mathbf{y},\bGamma) }
    = (\sigma'_{\mathbf{y}})^{*}\bGamma
    = (\mathbf{y} \times \bGamma, \mathbf{0}).
  \end{equation}
  Therefore, the Euler--Poincar\'e equation~\eqref{eq:EP} with advected parameters gives
  \begin{equation*}
    \begin{split}
    \od{}{t}\parentheses{ \pd{\ell}{\bOmega} } &= \ad_{\bOmega}^{*} \pd{\ell}{\bOmega} - \mathbf{J}\parentheses{\mathbf{v}, \pd{\ell}{\mathbf{v}}} + \mathbf{K}_{\so(3)^{*}}\parentheses{ \pd{\ell}{\bGamma}, \bGamma }, \\
    \od{}{t}\parentheses{ \pd{\ell}{\mathbf{v}} } &= \rho'(\bOmega)^{*} \pd{\ell}{\mathbf{v}} + \mathbf{K}_{(\R^{3})^{*}}\parentheses{ \pd{\ell}{\bGamma}, \bGamma }, \\
    \od{}{t}\bGamma &= \sigma'(\bOmega,\mathbf{v})^{*} \bGamma,
  \end{split}
  \end{equation*}
  or more concretely,
  \begin{equation}
    \label{eq:EP-uwv}
    \begin{split}
      \dot{\bPi} &= \bPi \times \bOmega + \bP \times \mathbf{v} - m \mathrm{g} l \boldsymbol{\chi} \times \bGamma, \\
      \dot{\bP} &= \bP \times \bOmega, \\
      \dot{\bGamma} &= \bGamma \times \bOmega
    \end{split}
  \end{equation}
  as in \cite{Le1997a,Le1997b,LeMa1997}.
\end{example}

\begin{example}[{Heavy top on movable base; see \citet{CoOh-EPwithBSym1}}]
  \label{ex:htmb}
  Consider the heavy top rotating on a movable base shown in Fig.~\ref{fig:htmb-details}.
  \begin{figure}[htbp]
    \centering
    \begin{tikzpicture}[scale=.875]
      \node[draw, cylinder, shape aspect=3, rotate=90, minimum height=1cm, minimum width=3cm] (c1)  at (0,1){};
      \draw[fill] (0,1.5) circle [radius=0.05];

      \fill[
      left color=gray!50!black,
      middle color=gray!50,
      right color=gray!50!black,
      shading=axis,
      opacity=0.25
      ] (3.75,3.75) -- (0,1.5) -- (1.3,5.2) [rotate=335] arc (180:340:1.49cm and 0.5cm); 
      \fill[
      shading=ball,
      ball color=gray!50,
      opacity=0.25
      ] (1.3,5.2) [rotate=335] arc (180:340:1.49cm and 0.5cm) arc (340:560:1.507cm); 

      \node at (1,5.5)    {$m$};
      \node at (2,0.8)    {$M$};

      \node[semitransparent] at (2,3.8) {\centerofmass};
      \draw[semithick,blue,->,>=stealth]  (0,1.5) -- (2,3.8);
      \node[left] at (1.9,3.7) {\color{blue}{$l\!\boldsymbol{\chi}$}};

      \draw[semithick,->,>=stealth]  (-5.5,1) -- (-5.5,1.5) node[above] {$\mathbf{e}_3$};
      \draw[semithick,->,>=stealth]  (-5.5,1) -- (-5.8,0.7) node[below] {$\mathbf{e}_1$};
      \draw[semithick,->,>=stealth]  (-5.5,1) -- (-5,1) node[below right] {$\mathbf{e}_2$};

      \draw[semithick,green!60!black][->,>=stealth]  (-5.5,1) -- (0,1.5);
      \node at (-3,1) {\color{green!60!black}{$\mathbf{x}$}};

      \draw[semithick,->,>=stealth]  (0,1.5) -- (-0.55,1.75) node[left] {$\mathbf{E}_1\!$};
      \draw[semithick,->,>=stealth]  (0,1.5) -- (0.29,1) node[right] {$\mathbf{E}_2$};
      \draw[semithick,->,>=stealth]  (0,1.5) -- (0.4,1.96) node[below right] {$\mathbf{E}_3$};

      \draw[very thick, ->,>=stealth, red!75!black] (-1.3,0.5) -- (-0.8,0.875);
      \node[below right, red!75!black] at (-1.3,0.575) {$\mathbf{u}$};

    \end{tikzpicture}
    \caption{Heavy top on a movable base.}
    \label{fig:htmb-details}
  \end{figure}
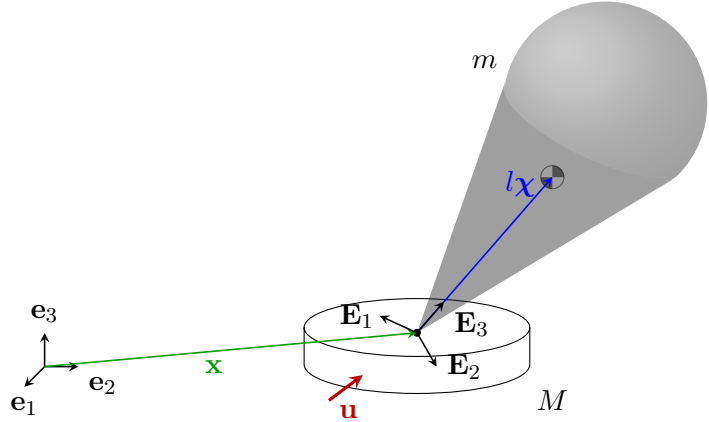
  The configuration space is again $\mathsf{S} = \SE(3)$, where the body frame is attached to the top at the junction point with the base (which is assumed to be a point mass $M$ for simplicity).
  The setting is almost the same as the underwater vehicle, and the kinetic energy is also defined in a similar manner.
    
  The only major difference is that the potential term depends not only on the orientation of the top but also on the height of the system:
  \begin{align*}
    U_{e_{3}}(R, \mathbf{x})
    &= m \mathrm{g} l \boldsymbol{\chi} \cdot  (R^{-1} \mathbf{e}_{3})
      + \bar{m} \mathrm{g} \mathbf{x} \cdot \mathbf{e}_{3} \\
    &= \mathrm{g}
      \begin{bmatrix}
        m l\!\boldsymbol{\chi} & \bar{m}
      \end{bmatrix}
                                \begin{bmatrix}
                                  R^{-1} & 0 \\
                                  \mathbf{x}^{T} & 1
                                \end{bmatrix}
                                                   \begin{bmatrix}
                                                     \mathbf{e}_{3} \\
                                                     0
                                                   \end{bmatrix} \\
    &= \mathrm{g} \mathfrak{m} \cdot \parentheses{ s^{T} e_{3} },
  \end{align*}
  where
  \begin{gather*}
    \bar{m} \defeq M + m,
    \qquad
    s = (R,\mathbf{x}) =
    \begin{bmatrix}
      R & \mathbf{x} \\
      \mathbf{0}^{T} & 1
    \end{bmatrix},
    \\
    \mathfrak{m} \defeq
    \begin{bmatrix}
      m l\!\boldsymbol{\chi} \\
      \bar{m}
    \end{bmatrix}
    \in \R^{4},
    \qquad
    e_{3} \defeq
    \begin{bmatrix}
      \mathbf{e}_{3} \\
      0
    \end{bmatrix}
      \in \R^{4}.
  \end{gather*}
  The potential $U_{e_{3}}$ is then clearly not $\SE(3)$-invariant.

  Let us define the extended potential $U\colon \SE(3) \times (\R^{4})^{*} \to \R$ by setting
  \begin{equation*}
    U(R, \mathbf{x}, a) \defeq \mathrm{g} \mathfrak{m} \cdot \parentheses{ s^{T} a }
  \end{equation*}
  so that $U(R, \mathbf{x}, e_{3}) = U_{e_{3}}(R, \mathbf{x})$.
  Also define the representation $\sigma\colon \SE(3) \to \GL(\R^{4})$ by
  \begin{equation}
    \label{eq:sigma-htmb}
    \sigma(s) y \defeq s y
    = \begin{bmatrix}
      R & \mathbf{x} \\
      \mathbf{0}^{T} & 1
    \end{bmatrix}
    \begin{bmatrix}
      \mathbf{y} \\
      y_{4}
    \end{bmatrix}
    = \begin{bmatrix}
      R \mathbf{y} + y_{4}\mathbf{x} \\
      y_{4}
    \end{bmatrix}.
  \end{equation}
  We note in passing that this representation was also used in the optimal-control formulation of the Kirchhoff elastic rod under gravity by \citet{BoBr2014,BoBr2016}.
  
  Identifying $(\R^{4})^{*}$ with $\R^{4}$ via the inner product, we have
  \begin{equation*}
    \parentheses{ \sigma^{*}(s) a } \cdot y
    = a \cdot \parentheses{ \sigma(s^{-1}) y }
    = a \cdot \parentheses{ s^{-1} y }
    = \parentheses{ (s^{T})^{-1} a } \cdot y.
  \end{equation*}
  Therefore, we have
  \begin{equation*}
    \sigma^{*}(s) a = (s^{T})^{-1} a.
  \end{equation*}
  As a result, we have, for any $s_{0}, s \in \SE(3)$,
  \begin{align*}
    U\parentheses{ s_{0} s, \sigma^{*}(s_{0}) a }
    &= \mathrm{g} \mathfrak{m} \cdot \parentheses{ (s_{0}s)^{T} (s_{0}^{T})^{-1} a } \\
    &= \mathrm{g} \mathfrak{m} \cdot \parentheses{ s^{T} a } \\
    &= U(s, a),
  \end{align*}
  that is, we have recovered the $\SE(3)$-symmetry.
  Therefore, writing $a = (\bGamma, h) \in (\R^{4})^{*}$---$h$ is the height of the base in the inertial frame---we may define the reduced potential $U\colon (\R^{4})^{*} \to \R$ as
  \begin{equation}
    \label{eq:U-htmb}
    U(\bGamma, h) = U(e, (\bGamma, h))
    = \mathrm{g} \mathfrak{m} \cdot (\bGamma, h)
    = m \mathrm{g} l \boldsymbol{\chi} \cdot \bGamma + \bar{m} \mathrm{g} h.
  \end{equation}

  Moreover,
  \begin{equation*}
    \sigma'(\bOmega,\mathbf{v}) y
    = \sigma_{y}'(\bOmega,\mathbf{v})
    = (\bOmega \times \mathbf{y} + y_{4}\mathbf{v},\, 0),
  \end{equation*}
  and thus
  \begin{align*}
    \parentheses{ \sigma'(\bOmega,\mathbf{v})^{*} a } \cdot y
    &= a \cdot \parentheses{ \sigma'(\bOmega,\mathbf{v}) y } \\
    &= (\bGamma, h) \cdot (\bOmega \times \mathbf{y} + y_{4}\mathbf{v}, 0) \\
    &= \bGamma \cdot (\bOmega \times \mathbf{y}) + y_{4}\bGamma \cdot \mathbf{v} \\
    &= (\bGamma \times \bOmega) \cdot \mathbf{y} + (\bGamma \cdot \mathbf{v}) y_{4},
  \end{align*}
  and so
  \begin{equation*}
    \sigma'(\bOmega,\mathbf{v})^{*} a
    = (\bGamma \times \bOmega,\, \bGamma \cdot \mathbf{v}).
  \end{equation*}
  Similarly,
  \begin{align*}
    \parentheses{ (\sigma'_{y})^{*}a } \cdot (\bOmega,\mathbf{v})
    &= a \cdot \parentheses{ \sigma'_{y}(\bOmega,\mathbf{v}) } \\
    &= \bGamma \cdot (\bOmega \times \mathbf{y}) + y_{4}\bGamma \cdot \mathbf{v} \\
    &= (\mathbf{y} \times \bGamma) \cdot \bOmega + (y_{4}\bGamma) \cdot \mathbf{v},
  \end{align*}
  and so we obtain, using \eqref{eq:K},
  \begin{equation}
    \label{eq:K-htmb}
    \mathbf{K}(y,a)
    = \parentheses{ \mathbf{K}_{\so(3)^{*}}(y,a), \mathbf{K}_{(\R^{3})^{*}}(y,a) }
    = (\sigma'_{y})^{*}a
    = (\mathbf{y} \times \bGamma, y_{4}\bGamma).
  \end{equation}
  Therefore, the Euler--Poincar\'e equation~\eqref{eq:EP} with advected parameters gives
  \begin{equation*}
    \begin{split}
    \od{}{t}\parentheses{ \pd{\ell}{\bOmega} } &= \ad_{\bOmega}^{*} \pd{\ell}{\bOmega} - \mathbf{J}\parentheses{\mathbf{v}, \pd{\ell}{\mathbf{v}}} + \mathbf{K}_{\so(3)^{*}}\parentheses{ \pd{\ell}{a}, a }, \\
    \od{}{t}\parentheses{ \pd{\ell}{\mathbf{v}} } &= \rho'(\bOmega)^{*} \pd{\ell}{\mathbf{v}} + \mathbf{K}_{(\R^{3})^{*}}\parentheses{ \pd{\ell}{a}, a }, \\
    \od{a}{t} &= \sigma'(\bOmega,\mathbf{v})^{*} a,
  \end{split}
  \end{equation*}
  or more concretely,
  \begin{equation*}
    \begin{split}
      \dot{\bPi} &= \bPi \times \bOmega + \bP \times \mathbf{v} - m \mathrm{g} l \boldsymbol{\chi} \times \bGamma, \\
      \dot{\bP} &= \bP \times \bOmega - \bar{m} \mathrm{g} \bGamma, \\
      \dot{\bGamma} &= \bGamma \times \bOmega, \\
      \dot{h} &= \bGamma \cdot \mathbf{v},
    \end{split}
  \end{equation*}
  as we have obtained in \cite{CoOh-EPwithBSym1}.
\end{example}

\subsection{Tracking Additional Advected Parameters}
\label{ssec:additional_advected_parameters}
In control applications of the Euler--Poincar\'e equation~\eqref{eq:EP} with advected parameters, one is often interested in tracking and stabilizing more variables in addition to the dynamical variables $(\xi, v, a) \in \mathfrak{s} \times V^{*}$.
Suppose that these additional variables $b$ live in the dual $Y^{*}$ of a vector space $Y$.
Being rather ancillary in nature, these variables can oftentimes be described as advected parameters via a representation $\tau\colon \mathsf{S} \to \GL(Y)$.
Note that it does not alter the equations of motion~\eqref{eq:EP}, i.e., one simply augments the equations of motion~\eqref{eq:EP} with
\begin{equation}
  \label{eq:b}
  \dot{b} = \tau'(\xi,v)^{*} b.
\end{equation}

\begin{example}[{Desired steady motion in underwater vehicle~\cite[Section~4.1]{Le1997b}}]
  \label{ex:desired_motion}
  Suppose that, for practical purposes, one would like to have the vehicle stay close to the desired orientation $R_{\rm d} \in \SO(3)$ and the desired velocity $\vd \in \R^{3}\backslash\{\mathbf{0}\}$ in the \textit{body} frame.
  
  The push-forward of the unit vector $\vd/\norm{ \vd }$ by $R_{\rm d}$ gives the \textit{fixed} unit vector $\mathbf{w}_{3} \defeq R_{\rm d} \vd/\norm{ \vd }$ in the \textit{spatial} frame.
  Then the pull-back of $\mathbf{w}_{3}$ by $R(t) \in \SO(3)$ gives the \textit{time-dependent} unit vector $\bTheta(t) \defeq R(t)^{T}\mathbf{w}_{3}$ in the \textit{body} frame at any time $t$; see Fig.~\ref{fig:uwv-desired_motion}.
  Then one can think of the deviation of $\bTheta(t)$ from $\vd/\norm{ \vd }$ as an indicator of deviation from the desired steady motion.

  \begin{figure}[hbtp]
    \centering
    \includegraphics[width=0.55\linewidth]{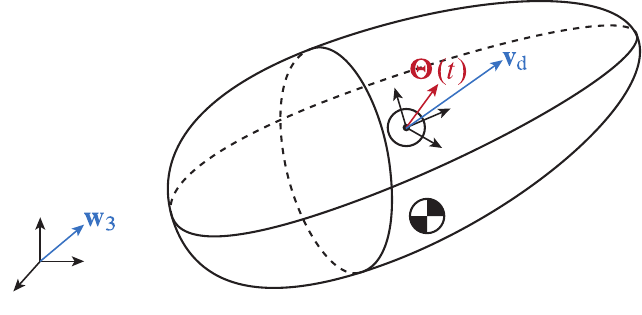}
    \caption{
      The desired velocity $\mathbf{v}_{\rm d}$ is a \textit{fixed} vector in the body frame, and $\mathbf{w}_{3} \defeq R_{\rm d} \vd/\norm{ \vd }$ is the \textit{fixed} unit vector in the spatial frame indicating the direction of the desired velocity when the vehicle is in the (fixed) desired orientation $R_{\rm d} \in \SO(3)$.
      On the other hand, $\bTheta(t) \defeq R(t)^{T}\mathbf{w}_{3}$ is the \textit{time-dependent} vector in the body frame that indicates the direction of $\mathbf{w}_{3}$ seen from the body frame, where $R(t) \in \SO(3)$ indicates the orientation of the vehicle at time $t$.}
    \label{fig:uwv-desired_motion}
  \end{figure}
  
  Specifically, if $R(0) = I$ then $\bTheta(t) \defeq R(t)^{T}\bTheta(0)$.
  This suggests that we set $Y = \R^{3}$ and define the representation $\tau\colon \SE(3) \to \GL(\R^{3})$ so that $\tau(R,\mathbf{x})^{*} \bTheta = R^{T} \bTheta$.
  In fact, defining $\tau(R,\mathbf{x}) \mathbf{y} = R \mathbf{y}$ would result in the desired expression.
  Note that $\tau$ is exactly the same representation as $\sigma$ from Example~\ref{ex:uwv}.
  As a result, we have $\tau'(\bOmega,\mathbf{v}) \mathbf{y} = \bOmega \times \mathbf{y}$, and so $\tau'(\bOmega,\mathbf{v})^{*} \bTheta = \bTheta \times \bOmega$ in view of \eqref{eq:sigma'-uwv} and \eqref{eq:sigma'star-uwv}.
  Hence the additional equation~\eqref{eq:b} becomes
  \begin{equation*}
    \dot{\bTheta} = \bTheta \times \bOmega.
  \end{equation*}
  
  As we shall see later, one may augment the Euler--Poincar\'e equation~\eqref{eq:EP-uwv} with the above equation to formulate the problem of finding a control to stabilize the direction of $\bTheta$, thereby achieving the stability of the desired steady motion.
\end{example}

\begin{example}[{Translational drift in underwater vehicle~\cite[Section~4.2]{Le1997b}}]
  \label{ex:drift}
  Suppose that, instead of tracking the desired velocity and orientation, one would like to track undesired drift of the underwater vehicle in those directions perpendicular to the direction $\mathbf{w}_{3}$ of the desired velocity in the \textit{spatial} frame.

  We show how to exploit representations and advected parameters to formulate the problem; this results in a more succinct formulation of the problem from \cite[Section~4.2]{Le1997b}.
  As we shall see in Example~\ref{ex:uwv-drift} below, our formulation still yields the same control law as that of \cite{Le1997b}.
  
  Let $\{ \mathbf{w}_{1}, \mathbf{w}_{2} \}$ be an orthonormal basis for $\Span\{ \mathbf{w}_{3} \}^{\perp}$ in the spatial frame defined so that $\{ \mathbf{w}_{1}, \mathbf{w}_{2}, \mathbf{w}_{3} \}$ is a right-handed system, i.e., $\mathbf{w}_{1} \times \mathbf{w}_{2} = \mathbf{w}_{3}$.
  Then the drift we would like to track (and would like to later prevent with controls) is $\boldsymbol{\delta}(t) = (\delta_{1}(t), \delta_{2}(t)) \defeq (\mathbf{x}(t) \cdot \mathbf{w}_{1}, \mathbf{x}(t) \cdot \mathbf{w}_{2}) \in \R^{2}$ at any time $t$; see Fig.~\ref{fig:uwv-drift}.
  
  \begin{figure}[hbtp]
    \centering
    \includegraphics[width=0.55\linewidth]{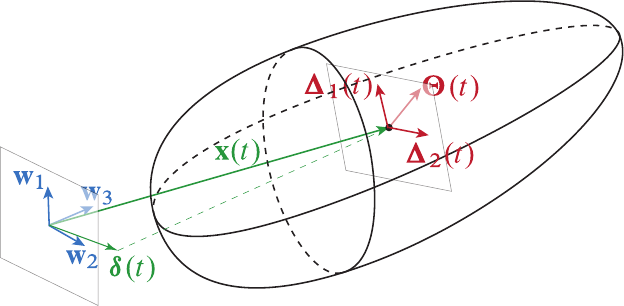}
    \caption{
      Fix an orthonormal basis $\{ \mathbf{w}_{1}, \mathbf{w}_{2} \}$ for $\Span\{ \mathbf{w}_{3} \}^{\perp}$, where $\mathbf{w}_{3}$ is defined in Fig.~\ref{fig:uwv-desired_motion}.
      Then $\bDelta_{i}(t) \defeq R(t)^{T} \mathbf{w}_{i}$ gives the directions of $\mathbf{w}_{i}$ seen from the body frame for $i = 1, 2$, whereas $\boldsymbol{\delta}(t) \defeq (\mathbf{x}(t) \cdot \mathbf{w}_{1}, \mathbf{x}(t) \cdot \mathbf{w}_{2})$ gives the undesired drift of the vehicle.
      The vector $\bTheta(t)$ is not tracked in this problem but is shown as a reference because $\{ \bDelta_{1}(t), \bDelta_{2}(t), \bTheta(t) \}$ defines an orthonormal basis for the body frame.
    }
    \label{fig:uwv-drift}
  \end{figure}
  
  Notice that, defining $\bDelta_{i}(t) \defeq R(t)^{T} \mathbf{w}_{i} \in \R^{3}$ and writing $\Delta_{i}(t) = (\bDelta_{i}(t), \delta_{i}(t)) \in \R^{4}$ for $i = 1, 2$, we have
  \begin{align*}
    \begin{bmatrix}
      \Delta_{1}(t) & \Delta_{2}(t)
    \end{bmatrix}
    = \begin{bmatrix}
      \bDelta_{1}(t) & \bDelta_{2}(t) \\
      \delta_{1}(t) & \delta_{2}(t)
    \end{bmatrix}
    &= \begin{bmatrix}
      R(t)^{T} \mathbf{w}_{1} &  R(t)^{T} \mathbf{w}_{2} \\
      \mathbf{x}(t) \cdot \mathbf{w}_{1} & \mathbf{x}(t) \cdot \mathbf{w}_{2}
    \end{bmatrix} \\
    &= \begin{bmatrix}
      R(t)^{T} & 0 \\
      \mathbf{x}(t)^{T} & 1
    \end{bmatrix}
    \begin{bmatrix}
      \mathbf{w}_{1} & \mathbf{w}_{2}\\
      0 & 0
    \end{bmatrix} \\
    &= s(t)^{T} W,
  \end{align*}
  where we wrote $s(t) = \begin{tbmatrix} R(t) & \mathbf{x}(t) \\ \mathbf{0}^{T} & 1 \end{tbmatrix}$ and $W \defeq \begin{tbmatrix}
    \mathbf{w}_{1} & \mathbf{w}_{2}\\
    0 & 0
  \end{tbmatrix}$.
  
  This suggests us to set $Y = \R^{4} \times \R^{4}$ and consider the representation $\tau\colon \SE(3) \to \GL(\R^{4} \times \R^{4})$ so that $\tau(s)^{*} (\Delta_{1}, \Delta_{2}) = (s^{T} \Delta_{1}, s^{T} \Delta_{2})$.
  In fact, we see that $\tau(s) (y, z) = (s y, s z)$ would do, and then since $\tau$ is two copies of the representation $\sigma$ from \eqref{eq:sigma-htmb}, we have, writing $y = (\mathbf{y}, y_{4})$ and $z = (\mathbf{z}, z_{4})$,
  \begin{equation}
    \label{eq:tau'-drift}
    \tau'(\bOmega,\mathbf{v}) (y, z)
    = \tau_{(y, z)}'(\bOmega,\mathbf{v})
    = \parentheses{    
      \begin{bmatrix}
        \bOmega \times \mathbf{y} + y_{4}\mathbf{v} \\
        0
      \end{bmatrix},
      \begin{bmatrix}
        \bOmega \times \mathbf{z} + z_{4}\mathbf{v} \\
        0
      \end{bmatrix}
    }.
  \end{equation}
  Therefore, we have
  \begin{equation*}
    \tau'(\bOmega,\mathbf{v})^{*} (\Delta_{1}, \Delta_{2}) =
    \parentheses{
      \begin{bmatrix}
        \bDelta_{1} \times \bOmega \\
        \bDelta_{1} \cdot \mathbf{v}
      \end{bmatrix},
      \begin{bmatrix}
        \bDelta_{2} \times \bOmega \\
        \bDelta_{2} \cdot \mathbf{v}
      \end{bmatrix}
    }.
  \end{equation*}
  Hence Eq.~\eqref{eq:b} for tracking additional variables becomes, for $i = 1, 2$,
  \begin{equation*}
    \dot{\bDelta}_{i} = \bDelta_{i} \times \bOmega,
    \qquad
    \dot{\delta}_{i} = \bDelta_{i} \cdot \mathbf{v}.
  \end{equation*}

  Our formulation is much simpler than that of \citet{Le1997b}; yet it turns out to be equivalent to hers.
  To see this, let us first set $Q \defeq [ \mathbf{w}_{1}\ \mathbf{w}_{2}\ \mathbf{w}_{3} ]^{T}$.
  Then $Q \in \SO(3)$ because $\{ \mathbf{w}_{1}, \mathbf{w}_{2}, \mathbf{w}_{3} \}$ is a right-handed orthonormal basis.
  Therefore, we have $Q Q^{T} = I$ or $Q \mathbf{w}_{i} = \mathbf{e}_{i}$ for $i = 1, 2, 3$; specifically $Q R_{\rm d} \vd = \norm{\vd}\,\mathbf{e}_{3}$, i.e., this is the same $Q$ defined in \cite{Le1997b}.
  Then
  \begin{equation*}
    \delta_{i} = \mathbf{x}^{T} \mathbf{w}_{i} = \mathbf{x}^{T} Q^{T} \mathbf{e}_{i} = (Q\mathbf{x})^{T} \mathbf{e}_{i}.
  \end{equation*}
  So $\boldsymbol{\delta} = (\delta_{1}, \delta_{2})$ gives the first two components of $Q\mathbf{x}$; but then this implies that $\boldsymbol{\delta}$ is nothing but the first two components of $\tilde{b}$ in \cite{Le1997b}.
  Notice that our evolution equation for $\boldsymbol{\delta}$ is much simpler than that for $\tilde{b}$ in \cite{Le1997b}.
  This demonstrates the advantage of our geometric approach using representations and advected parameters.
  We shall continue the comparison in Example~\ref{ex:uwv-drift} below to show that our formulation yields the same control law in a simpler form as well.
\end{example}

\section{Potential Shaping and Matching Conditions}
\label{sec:potential_shaping}
\subsection{Controlled Euler--Poincar\'e Equation with Advected Parameters}
Let $\mathcal{I} \subset \R$ be the time interval of interest, and apply control $u = (u_{\mathfrak{g}^{*}}, u_{V^{*}}) \colon \mathcal{I} \to \mathfrak{g}^{*} \times V^{*}$ to the Euler--Poincar\'e equation~\eqref{eq:EP} augmented with the equation~\eqref{eq:b} for additional variables $b$ to track, i.e.,
\begin{equation}
  \label{eq:ControlledEP-1}
  \begin{split}
    \od{}{t}\parentheses{ \fd{\ell}{\xi} } &= \ad_{\xi}^{*} \fd{\ell}{\xi} - \mathbf{J}\parentheses{v, \fd{\ell}{v}} + \mathbf{K}_{\mathfrak{g}^{*}}\parentheses{ \fd{\ell}{a}, a } + u_{\mathfrak{g}^{*}}, \\
    \od{}{t}\parentheses{ \fd{\ell}{v} } &= \rho'(\xi)^{*} \fd{\ell}{v} + \mathbf{K}_{V^{*}}\parentheses{ \fd{\ell}{a}, a } + u_{V^{*}}
  \end{split}
\end{equation}
coupled with either just
\begin{subequations}
  \begin{equation}
    \label{eq:ControlledEP-2a}
    \od{a}{t} = \sigma'(\xi,v)^{*} a,
  \end{equation}
  or, with additional advected parameters $b$ to track,
  \begin{equation}
    \label{eq:ControlledEP-2b}
    \od{a}{t} = \sigma'(\xi,v)^{*} a,
    \qquad
    \od{b}{t} = \tau'(\xi,v)^{*} b.
  \end{equation}
\end{subequations}
We note that the control is applied only to the $\mathfrak{g}^{*} \times V^{*}$-part of the equation, not to the $X^{*}$-part for the advected parameters.

Our goal is to find a control that stabilizes an equilibrium of the above set of equations.
The first step towards the goal is the matching using the controlled Lagrangian considered in various settings~\cite{Ha1999,Ha2000,BlChLeMa2001,BlLeMa1999b,BlLeMa2000,BlLeMa2001,ChBlLeMaWo2002,ChMa2004,BlOrVa2002,OrPeNiSi1998,OrScMaMa2001}.
More specifically, we would like to find a new Lagrangian $\tilde{\ell}$---called the \textit{controlled Lagrangian}---such that its corresponding \textit{uncontrolled} Euler--Poincar\'e equation becomes identical to the original \textit{controlled} system~\eqref{eq:ControlledEP-1} with \eqref{eq:ControlledEP-2a} or \eqref{eq:ControlledEP-2b}.

We are particularly interested in the \textit{potential shaping}, i.e., we seek the new Lagrangian $\tilde{\ell}$ by changing only the potential term in the original Lagrangian $\ell$ in \eqref{eq:ell}.

In the sections to follow, we will show two different types of matching via potential shaping.
In Section~\ref{ssec:matching-1}, we will show how one can reduce the equation~\eqref{eq:ControlledEP-2a} using a subrepresentation so that the controlled system becomes the Euler--Poincar\'e equation involving \textit{less} advected parameters.
In Section~\ref{ssec:matching-2}, we will show how to incorporate the additional advected parameter $b$ in \eqref{eq:ControlledEP-2b} into the control so that one can achieve stability in the system involving \textit{more} advected parameters.

\subsection{Matching via Potential Shaping I: Reducing to Subrepresentation}
\label{ssec:matching-1}
In our previous work~\cite{CoOh-EPwithBSym1} on the heavy top on a movable base from Example~\ref{ex:htmb} (see also Example~\ref{ex:htmb-matching} below), we used an ad-hoc potential shaping to slightly simplify the system, and then applied a kinetic shaping to stabilize an equilibrium of the system.
We generalize this idea in this subsection.

Suppose that $\sigma\colon \mathsf{S} \to \GL(X)$ has a subrepresentation $\sigma|_{\tilde{X}}\colon \mathsf{S} \to \GL(\tilde{X})$ with some subspace $\tilde{X} \subset X$.
Let $\tilde{\mathbf{K}}\colon \tilde{X} \times \tilde{X}^{*} \to \mathfrak{s}^{*}$ be the corresponding momentum map, and as a result, the Euler--Poincar\'e equation with the Lagrangian $\tilde{\ell}\colon \mathfrak{s} \times \tilde{X}^{*} \to \R$ becomes
\begin{equation}
  \label{eq:EP-tilde-I}
  \begin{split}
    \od{}{t}\parentheses{ \fd{\tilde{\ell}}{\xi} } &= \ad_{\xi}^{*} \fd{\tilde{\ell}}{\xi} - \mathbf{J}\parentheses{v, \fd{\tilde{\ell}}{v}} + \tilde{\mathbf{K}}_{\mathfrak{g}^{*}}\parentheses{ \fd{\tilde{\ell}}{\tilde{a}}, \tilde{a} }, \\
    \od{}{t}\parentheses{ \fd{\tilde{\ell}}{v} } &= \rho'(\xi)^{*} \fd{\tilde{\ell}}{v} + \tilde{\mathbf{K}}_{V^{*}}\parentheses{ \fd{\tilde{\ell}}{\tilde{a}}, \tilde{a} }, \\
    \od{\tilde{a}}{t} &= \sigma'(\xi,v)^{*} \tilde{a}.
  \end{split}
\end{equation}
Now our goal is the matching between the controlled equations \eqref{eq:ControlledEP-1} with \eqref{eq:ControlledEP-2a} and the above equation~\eqref{eq:EP-tilde-I}.
Note however that this is not a strict equivalence; it rather effectively discards some components of the original advected parameters:

\begin{theorem}[Matching via Subrepresentation]
  Let $\ell\colon \mathfrak{s} \times X^{*} \to \R$ be the Lagrangian defined in \eqref{eq:ell}, and $\tilde{\ell}\colon \mathfrak{s} \times \tilde{X}^{*} \to \R$ be the controlled Lagrangian defined with a modified potential $\tilde{U}\colon \tilde{X}^{*} \to \R$ as
  \begin{equation}
    \label{eq:tilde_ell-I}
    \tilde{\ell}(\xi, v, \tilde{a}) =  \frac{1}{2} \metric{(\xi,v)}{(\xi,v)} - \tilde{U}(\tilde{a}).
  \end{equation}
  The controlled Euler--Poincar\'e equation~\eqref{eq:ControlledEP-1} and \eqref{eq:ControlledEP-2a} match the Euler--Poincar\'e equation~\eqref{eq:EP-tilde-I} if and only if the control $u = (u_{\mathfrak{g}^{*}}, u_{V^{*}})$ and the potential $U$ satisfy
  \begin{equation}
    \label{eq:u-I}
    \begin{split}
      u_{\mathfrak{g}^{*}} &= \tilde{\mathbf{K}}_{\mathfrak{g}^{*}}\parentheses{ \fd{\tilde{U}}{\tilde{a}}, \tilde{a} } - \mathbf{K}_{\mathfrak{g}^{*}}\parentheses{ \fd{U}{a}, a }, \\
      u_{V^{*}} &= \mathbf{K}_{V^{*}}\parentheses{ \fd{U}{a}, a } - \tilde{\mathbf{K}}_{V^{*}}\parentheses{ \fd{\tilde{U}}{\tilde{a}}, \tilde{a} }.
    \end{split}
  \end{equation}
\end{theorem}

\begin{proof}
  We see that $\tfd{\tilde{\ell}}{\xi} = \tfd{\ell}{\xi}$ and $\tfd{\tilde{\ell}}{v} = \tfd{\ell}{v}$, and thus the matching is achieved if and only if the control $u = (u_{\mathfrak{g}^{*}}, u_{V^{*}})$ and $\tilde{\mathbf{K}}$ satisfy
  \begin{equation*}
    \mathbf{K}_{\mathfrak{g}^{*}}\parentheses{ \fd{\ell}{a}, a } + u_{\mathfrak{g}^{*}} = \tilde{\mathbf{K}}_{\mathfrak{g}^{*}}\parentheses{ \fd{\tilde{\ell}}{\tilde{a}}, \tilde{a} },
    \qquad
    \mathbf{K}_{V^{*}}\parentheses{ \fd{\ell}{a}, a } + u_{V^{*}} = \tilde{\mathbf{K}}_{V^{*}}\parentheses{ \fd{\tilde{\ell}}{\tilde{a}}, \tilde{a} }.
  \end{equation*}
  These conditions reduce to \eqref{eq:u-I} in view of the Lagrangians \eqref{eq:ell} and \eqref{eq:tilde_ell-I}. \qed
\end{proof}


We note that the matching condition~\eqref{eq:u-I} implies that the expressions for $\mathbf{K}_{\mathfrak{g}^{*}}\parentheses{ \tfd{U}{a}, a }$ and $\mathbf{K}_{V^{*}}\parentheses{ \tfd{U}{a}, a }$ contain variables $\tilde{a}$ only.
Although this is rather restrictive, it is what happens in the example from our previous work mentioned above:

\begin{example}[{Potential shaping for heavy top on movable base~\cite{CoOh-EPwithBSym1}}]
  \label{ex:htmb-matching}
  Let us apply controls to the system from Example~\ref{ex:htmb}:
  \begin{equation}
    \label{eq:ControlledEP-htmb}
    \begin{split}
      \dot{\bPi} &= \bPi \times \bOmega + \bP \times \mathbf{v} - m \mathrm{g} l \boldsymbol{\chi} \times \bGamma, \\
      \dot{\bP} &= \bP \times \bOmega - \bar{m} \mathrm{g} \bGamma + \mathbf{u}_{(\R^{3})^{*}}, \\
      \dot{\bGamma} &= \bGamma \times \bOmega, \\
      \dot{h} &= \bGamma \cdot \mathbf{v},
    \end{split}
  \end{equation}
  Note that we do not have any control in the first set of equations, i.e., $\mathbf{u}_{\so(3)^{*}} = \mathbf{0}$ because we would like to stabilize the system by applying controls only to the base (not to the top); see Fig.~\ref{fig:htmb-details}.
  
  Then the matching conditions~\eqref{eq:u-I} become
  \begin{equation*}
    \mathbf{K}_{\so(3)^{*}}\parentheses{ \fd{U}{a}, a } = \tilde{\mathbf{K}}_{\so(3)^{*}}\parentheses{ \fd{\tilde{U}}{\tilde{a}}, \tilde{a} },
    \quad
    \mathbf{K}_{(\R^{3})^{*}}\parentheses{ \fd{U}{a}, a } - \mathbf{u}_{(\R^{3})^{*}} = \tilde{\mathbf{K}}_{(\R^{3})^{*}}\parentheses{ \fd{\tilde{U}}{\tilde{a}}, \tilde{a} },
  \end{equation*}
  where $a = (\bGamma,h)$.
  Using \eqref{eq:K-htmb}, they give
  \begin{equation*}
    \tilde{\mathbf{K}}_{\so(3)^{*}}\parentheses{ \fd{\tilde{U}}{\tilde{a}}, \tilde{a} } = \pd{U}{\bGamma} \times \bGamma,
    \quad
    \tilde{\mathbf{K}}_{(\R^{3})^{*}}\parentheses{ \fd{\tilde{U}}{\tilde{a}}, \tilde{a} } = \pd{U}{h} \bGamma - \mathbf{u}_{(\R^{3})^{*}}
    = \bar{m} \mathrm{g} \bGamma - \mathbf{u}_{(\R^{3})^{*}}.
  \end{equation*}
  
  The first condition suggests us to use the subrepresentation of $\sigma$ (see \eqref{eq:sigma-htmb}) on $\R^{3}$:
  \begin{equation*}
    \tilde{\sigma}\colon \SE(3) \to \GL(\R^{3});
    \qquad
    \tilde{\sigma}(R,\mathbf{x}) \mathbf{y} \defeq R\mathbf{y},
  \end{equation*}
  because this implies that one should take $\tilde{a} = \bGamma$; note that this is the $\sigma$ used for the underwater vehicle in Example~\ref{ex:uwv}.
  In fact, the corresponding momentum map $\tilde{\mathbf{K}}$ would be the same as \eqref{eq:K-uwv}:
  \begin{equation*}
    \tilde{\mathbf{K}}(\mathbf{y},\bGamma)
    = \parentheses{ \tilde{\mathbf{K}}_{\so(3)^{*}}(\mathbf{y},\bGamma), \tilde{\mathbf{K}}_{(\R^{3})^{*}}(\mathbf{y},\bGamma) }
    = (\mathbf{y} \times \bGamma, \mathbf{0}),
  \end{equation*}
  and so the first matching conditions yields
  \begin{equation*}
    \pd{U}{\bGamma} \times \bGamma = \pd{\tilde{U}}{\bGamma} \times \bGamma,
    \qquad
    \mathbf{u}_{(\R^{3})^{*}} = \bar{m} \mathrm{g} \bGamma.
  \end{equation*}
  This suggests us to define the new potential $\tilde{U}$ in terms of the original one $U$ from \eqref{eq:U-htmb} as follows:
  \begin{equation*}
    \tilde{U}(\bGamma) \defeq U(\bGamma,0) = m \mathrm{g} l \boldsymbol{\chi} \cdot \bGamma.
  \end{equation*}
  As a result, the controlled system~\eqref{eq:ControlledEP-htmb} becomes
  \begin{equation*}
    \begin{split}
      \dot{\bPi} &= \bPi \times \bOmega + \bP \times \mathbf{v} - m \mathrm{g} l \boldsymbol{\chi} \times \bGamma, \\
      \dot{\bP} &= \bP \times \bOmega, \\
      \dot{\bGamma} &= \bGamma \times \bOmega.
    \end{split}
  \end{equation*}
  So we effectively dropped the height $h$ from the formulation, and also that this is an Euler--Poincar\'e equation on $\se(3) \times (\R^{3})^{*}$ as opposed to $\se(3) \times (\R^{4})^{*}$.

  We note that we need to apply additional control to the base to stabilize the upright spinning position.
  This was done by kinetic shaping in the companion paper \cite{CoOh-EPwithBSym1} after applying the above potential shaping.
  
  The above potential shaping is rather simple in hindsight: It is simply applying the force to the base to cancel the gravitational force.
  However, it has an important implication that the system after the potential shaping has one more Casimir than the original system because the original system is defined on $\se(3) \times (\R^{4})^{*}$ whereas the new system on $\se(3) \times (\R^{3})^{*}$.
  One can then apply the kinetic shaping to the new system maintaining the new Casimir as an invariant.
  This facilitates the use of the energy-Casimir method; see \cite{CoOh-EPwithBSym1} for details.
\end{example}

\subsection{Matching via Potential Shaping II: With Additional Variables}
\label{ssec:matching-2}
In Section~\ref{ssec:additional_advected_parameters}, we showed how to track additional advected parameters.
In practical control problems, the equilibrium to stabilize is sometimes better characterized in terms of those advected parameters in addition to the original variables.
In this subsection, we continue our discussion from Section~\ref{ssec:additional_advected_parameters} to formulate a matching condition that applies to such settings.

The idea is to find an alternative form of Lagrangian $\tilde{\ell}\colon \mathfrak{s} \times X^{*} \times Y^{*} \to \R$ such that the corresponding Euler--Poincar\'e equation matches with \eqref{eq:ControlledEP-1} along with \eqref{eq:ControlledEP-2b}.
Note that the Lagrangian $\tilde{\ell}$ now depends on the additional variables $b$ in $Y^{*}$ as well.
Therefore, the Euler--Poincar\'e equation is now coupled with the equation~\eqref{eq:b} for $b$:
\begin{equation}
  \label{eq:EP-tilde-II}
  \begin{split}
    \od{}{t}\parentheses{ \fd{\tilde{\ell}}{\xi} } &= \ad_{\xi}^{*} \fd{\tilde{\ell}}{\xi} - \mathbf{J}\parentheses{v, \fd{\tilde{\ell}}{v}} + \mathbf{K}_{\mathfrak{g}^{*}}\parentheses{ \fd{\tilde{\ell}}{a}, a } + \mathbf{M}_{\mathfrak{g}^{*}}\parentheses{ \fd{\tilde{\ell}}{b}, b }, \\
    \od{}{t}\parentheses{ \fd{\tilde{\ell}}{v} } &= \rho'(\xi)^{*} \fd{\tilde{\ell}}{v} + \mathbf{K}_{V^{*}}\parentheses{ \fd{\tilde{\ell}}{a}, a } + \mathbf{M}_{V^{*}}\parentheses{ \fd{\tilde{\ell}}{b}, b }
  \end{split}
\end{equation}
along with \eqref{eq:ControlledEP-2b}, where we defined the momentum map $\mathbf{M}\colon Y \times Y^{*} \to \mathfrak{s}^{*}$ corresponding to the representation $\tau\colon \mathsf{S} \to \GL(Y)$ as
\begin{equation*}
  \mathbf{M}(y, b)
  = \parentheses{ \mathbf{M}_{\mathfrak{g}^{*}}(y, b),\, \mathbf{M}_{V^{*}}(y, b) }
  \defeq (\tau_{y}')^{*}b,
\end{equation*}
just like how we defined the momentum map $\mathbf{K}$ in Section~\ref{ssec:recovery}.

\begin{theorem}[Matching with Additional Variables]
  Let $\ell\colon \mathfrak{s} \times X^{*} \to \R$ be the Lagrangian defined in \eqref{eq:ell}, and $\tilde{\ell}\colon \mathfrak{s} \times \tilde{X}^{*} \times \tilde{Y}^{*} \to \R$ be the controlled Lagrangian that differs from $\ell$ by the additional potential term $\tilde{U}\colon X^{*} \times Y^{*} \to \R$, i.e.,
  \begin{equation*}
    \tilde{\ell}(\xi, v, a, b)
    = \ell(\xi, v, a) - \tilde{U}(a,b)
    = \frac{1}{2} \metric{(\xi,v)}{(\xi,v)} - U(a) - \tilde{U}(a,b).
  \end{equation*}
  The controlled Euler--Poincar\'e equations~\eqref{eq:ControlledEP-1} and \eqref{eq:ControlledEP-2b} match the Euler--Poincar\'e equations~\eqref{eq:EP-tilde-II} and \eqref{eq:ControlledEP-2b} if and only if the control $u$ and the additional potential term $\tilde{U}$ satisfy
  \begin{equation}
    \label{eq:u-II}
    u = -\mathbf{K}\parentheses{ \fd{\tilde{U}}{a}, a } - \mathbf{M}\parentheses{ \fd{\tilde{U}}{b}, b }.
  \end{equation}
\end{theorem}

\begin{proof}
  Equations~\eqref{eq:ControlledEP-2b} for the advected parameters $a$ and $b$ are the same in both sets of equations because they do not depend on the Lagrangian.
  Therefore it boils down to the matching between \eqref{eq:ControlledEP-1} and \eqref{eq:EP-tilde-II}.

  Clearly $\tfd{\tilde{\ell}}{\xi} = \tfd{\ell}{\xi}$ and $\tfd{\tilde{\ell}}{v} = \tfd{\ell}{v}$, and thus \eqref{eq:ControlledEP-1} and \eqref{eq:EP-tilde-II} match if and only if
  \begin{equation*}
    \mathbf{K}\parentheses{ \fd{\tilde{\ell}}{a}, a } + \mathbf{M}\parentheses{ \fd{\tilde{\ell}}{b}, b }
    = \mathbf{K}\parentheses{ \fd{\ell}{a}, a } + u,
  \end{equation*}
  or equivalently
  \begin{equation*}
    u
    = \mathbf{K}\parentheses{ \fd{\tilde{\ell}}{a} - \fd{\ell}{a}, a } + \mathbf{M}\parentheses{ \fd{\tilde{\ell}}{b}, b }
    = -\mathbf{K}\parentheses{ \fd{\tilde{U}}{a}, a } - \mathbf{M}\parentheses{ \fd{\tilde{U}}{b}, b }. \ \qed
  \end{equation*}
\end{proof}

Our goal is to find controls $u$ that stabilize those equilibria of the controlled system that would be either unstable or not even equilibria if uncontrolled.
This step imposes more concrete conditions on the potential $\tilde{U}$ so that one can determine explicit feedback control $u$.

To demonstrate the above result, let us show that it gives a unified framework for the two stabilization problems from \citet{Le1997b}:

\begin{example}[{Stabilizing underwater vehicle with desired steady motion~\cite[Section~4.1]{Le1997b}}]
  \label{ex:uwv-desired_motion}
  Continuing from Example~\ref{ex:desired_motion}, consider the problem of controlling the underwater vehicle with a desired steady motion:
  \begin{equation}
    \label{eq:ControlledEP-uwv-desired_motion}
    \begin{split}
      \dot{\bPi} &= \bPi \times \bOmega + \bP \times \mathbf{v} - m \mathrm{g} l \boldsymbol{\chi} \times \bGamma + \mathbf{u}_{\so(3)^{*}}, \\
      \dot{\bP} &= \bP \times \bOmega + \mathbf{u}_{(\R^{3})^{*}}, \\
      \dot{\bGamma} &= \bGamma \times \bOmega, \\
      \dot{\bTheta} &= \bTheta \times \bOmega,
    \end{split}
  \end{equation}
  where the equilibrium is given in terms of the desired orientation $R_{\rm d}$ and the desired velocity $\vd$ in the body frame (see Example~\ref{ex:desired_motion}) as follows:
  \begin{equation*}
    \zeta_{\rm e} \defeq (\bOmega_{\rm e}, \mathbf{v}_{\rm e}, \bGamma_{\rm e}, \bTheta_{\rm e})
    = \parentheses{
      \mathbf{0}, \vd, R_{\rm d}^{T}\mathbf{e}_{3}, \vd/\norm{ \vd }
      }.
  \end{equation*}
  It corresponds to the steady motion at the constant velocity $R_{\rm d}\mathbf{v}_{\rm d}$ in the spatial frame in the fixed attitude where the center of mass is right below the center of buoyancy.

  Note that $\mathbf{K}$ is given in \eqref{eq:K-uwv}.
  For $\mathbf{M}$, recall from Example~\ref{ex:desired_motion} that the representation $\tau$ is identical to $\sigma$ from Example~\ref{ex:uwv}.
  Therefore, $\mathbf{M}$ here is the same as $\mathbf{K}$ in \eqref{eq:K-uwv} from Example~\ref{ex:uwv}:
  \begin{equation*}
    \mathbf{M}(\mathbf{y}, \bTheta)
    = \parentheses{ \mathbf{M}_{\so(3)^{*}}(\mathbf{y}, \bTheta),\, \mathbf{M}_{(\R^{3})^{*}}(\mathbf{y}, \bTheta) }
    = (\tau_{\mathbf{y}}')^{*} \bTheta
    = (\mathbf{y} \times \bTheta, \mathbf{0}),
  \end{equation*}
  which is identical to $\mathbf{K}$; this is because $X = Y = \R^{3}$ and the representations $\sigma$ and $\tau$ are identical.
  As a result, \eqref{eq:u-II} yields
  \begin{equation}
    \label{eq:u-uwv-desired_motion}
    u(\bGamma, \bTheta)
    = \parentheses{ \mathbf{u}_{\so(3)^{*}}, \mathbf{u}_{(\R^{3})^{*}} }
    = \parentheses{
      -\pd{\tilde{U}}{\bGamma} \times \bGamma - \pd{\tilde{U}}{\bTheta} \times \bTheta,\,
      \mathbf{0}
    }.
  \end{equation}

  Now that we have the controlled system~\eqref{eq:ControlledEP-uwv-desired_motion} with control~\eqref{eq:u-uwv-desired_motion}, we would like to find a control $u$ that renders $\zeta_{\rm e}$ a stable equilibrium.
  The corresponding angular and linear impulses are $(\bPi_{\rm e}, \bP_{\rm e}) = (D \vd, M \vd)$, where $D$ and $M$ are from the kinetic energy metric~\eqref{eq:metric-uwv}.
  Note that $(\bOmega_{\rm e}, \mathbf{v}_{\rm e}, \bGamma_{\rm e})$ is \textit{not} an equilibrium of the \textit{uncontrolled} system~\eqref{eq:EP-uwv}.
  We would like to show that it is a stable equilibrium of the \textit{controlled} system so that the desired steady motion $\zeta_{\rm e}$ of the controlled system becomes stable.

  The point $\zeta_{\rm e}$ is an equilibrium of the controlled system~\eqref{eq:ControlledEP-uwv-desired_motion} with control~\eqref{eq:u-uwv-desired_motion} if and only if
  \begin{align*}
    \mathbf{u}_{\so(3)^{*}}(\bGamma_{\rm e}, \bTheta_{\rm e})
    &= m \mathrm{g} l \boldsymbol{\chi} \times \bGamma_{\rm e} - \bP_{\rm e} \times \mathbf{v}_{\rm e} \\
    &= m \mathrm{g} l \boldsymbol{\chi} \times \bGamma_{\rm e} - (M \vd) \times \vd,
  \end{align*}
  whereas the matching condition \eqref{eq:u-uwv-desired_motion} yields
  \begin{equation*}
    \mathbf{u}_{\so(3)^{*}}(\bGamma_{\rm e}, \bTheta_{\rm e}) =
    -\pd{\tilde{U}}{\bGamma}(\bGamma_{\rm e}, \bTheta_{\rm e}) \times \bGamma_{\rm e}
    - \frac{1}{\norm{\vd}} \pd{\tilde{U}}{\bTheta}(\bGamma_{\rm e}, \bTheta_{\rm e}) \times \vd.
  \end{equation*}
  Therefore, one can achieve matching by requiring $\tilde{U}$ to satisfy
  \begin{equation*}
    \pd{\tilde{U}}{\bGamma}(\bGamma_{\rm e}, \bTheta_{\rm e}) = -m \mathrm{g} l (\boldsymbol{\chi} + \beta \bGamma_{\rm e}),
    \qquad
    \pd{\tilde{U}}{\bTheta}(\bGamma_{\rm e}, \bTheta_{\rm e}) = \norm{\vd} (M - \alpha I) \vd
  \end{equation*}
  with arbitrary constants $\alpha, \beta \in \R$.
  The simplest form of $\tilde{U}$ that satisfies these conditions would be
  \begin{equation*}
    \tilde{U}(\bGamma, \bTheta)
    = -m \mathrm{g} l (\boldsymbol{\chi} + \beta \bGamma_{\rm e}) \cdot \bGamma
    + \norm{\vd} ((M - \alpha I) \vd) \cdot \bTheta.
  \end{equation*}
  As a result, we obtain the control
  \begin{equation*}
    \mathbf{u}_{\so(3)^{*}}(\bGamma, \bTheta)
    = m \mathrm{g} l (\boldsymbol{\chi} + \beta \bGamma_{\rm e}) \times \bGamma
    + \bTheta \times  (\norm{\vd} (M - \alpha I) \vd),
    \quad
    \mathbf{u}_{(\R^{3})^{*}} = \mathbf{0},
  \end{equation*}
  which are exactly Eq.~(4) of \cite{Le1997b}; note that her $r$ is our $-\boldsymbol{\chi}$.
  It is shown in \cite[Theorem~4.2]{Le1997b} using the energy--Casimir method that this control indeed stabilizes the equilibrium $\zeta_{\rm e}$ if $\alpha$ and $\beta$ satisfy $\alpha l > M$ and $l \beta > 0$.
\end{example}

\begin{example}[{Preventing drift in underwater vehicle~\cite[Section~4.2]{Le1997b}}]
  \label{ex:uwv-drift}
  Continuing from Example~\ref{ex:drift}, consider the problem of controlling the underwater vehicle with a particular interest in preventing undesired drift:
  \begin{equation}
    \label{eq:ControlledEP-uwv-drift}
    \begin{split}
      \dot{\bPi} &= \bPi \times \bOmega + \bP \times \mathbf{v} - m \mathrm{g} l \boldsymbol{\chi} \times \bGamma + \mathbf{u}_{\so(3)^{*}}, \\
      \dot{\bP} &= \bP \times \bOmega + \mathbf{u}_{(\R^{3})^{*}}, \\
      \dot{\bGamma} &= \bGamma \times \bOmega, \\
      \dot{\bDelta}_{i} &= \bDelta_{i} \times \bOmega, \\
      \dot{\delta}_{i} &= \bDelta_{i} \cdot \mathbf{v}
    \end{split}
  \end{equation}
  with $i = 1, 2$.
  As discussed in Example~\ref{ex:drift}, $\boldsymbol{\delta} = (\delta_{1}, \delta_{2})$ gives the undesired drift.
  
  Note that $\mathbf{K}$ is given in \eqref{eq:K-uwv}.
  Let us find $\mathbf{M}$.
  Using \eqref{eq:tau'-drift}, we find, for any $y = (\mathbf{y}, y_{4}), z = (\mathbf{z}, z_{4}) \in \R^{4}$ and any $\Delta_{i} = (\bDelta_{i}, \delta_{i}) \in \R^{4}$ with $i = 1, 2$,
  \begin{align*}
    \ip{ (\tau_{(y, z)}')^{*}(\Delta_{1}, \Delta_{2}) }{ (\bOmega,\mathbf{v}) }
    &= \ip{ (\Delta_{1}, \Delta_{2}) }{ \tau_{(y, z)}'(\bOmega,\mathbf{v}) } \\
    &= \bDelta_{1} \cdot (\bOmega \times \mathbf{y} + y_{4}\mathbf{v}) 
      + \bDelta_{2} \cdot (\bOmega \times \mathbf{z} + z_{4}\mathbf{v}) \\
    &= (\mathbf{y} \times \bDelta_{1} + \mathbf{z} \times \bDelta_{2}) \cdot \bOmega + (y_{4} \bDelta_{1}  + z_{4} \bDelta_{2}) \cdot \mathbf{v},
  \end{align*}
  and so
  \begin{equation*}
    (\tau_{(y,z)}')^{*}(\Delta_{1}, \Delta_{2}) = (\mathbf{y} \times \bDelta_{1} + \mathbf{z} \times \bDelta_{2},\,  y_{4} \bDelta_{1}  + z_{4} \bDelta_{2}).
  \end{equation*}
  Hence we obtain the momentum map $\mathbf{M}\colon (\R^{4} \times \R^{4}) \times (\R^{4} \times \R^{4})^{*} \to \se(3)^{*}$ as follows:
  \begin{align*}
    \mathbf{M}( (y, z), (\Delta_{1}, \Delta_{2}) )
    &= \parentheses{ \mathbf{M}_{\so(3)^{*}}( (y, z), (\bDelta_{1}, \bDelta_{2}) ),\, \mathbf{M}_{(\R^{3})^{*}}( (y, z), (\bDelta_{1}, \bDelta_{2}) ) } \\
    &= (\tau_{(y,z)}')^{*}(\bDelta_{1}, \bDelta_{2}) \\
    &= (\mathbf{y} \times \bDelta_{1} + \mathbf{z} \times \bDelta_{2},\,  y_{4} \bDelta_{1}  + z_{4} \bDelta_{2}).
  \end{align*}
  As a result, \eqref{eq:u-II} yields
  \begin{equation}
    \label{eq:u-uwv-drift}
    \begin{split}
      u(\bGamma, \Delta_{1}, \Delta_{2})
      &= \parentheses{ \mathbf{u}_{\so(3)^{*}}, \mathbf{u}_{(\R^{3})^{*}} } \\
      &= \parentheses{
        -\pd{\tilde{U}}{\bGamma} \times \bGamma - \pd{\tilde{U}}{\bDelta_{1}} \times \bDelta_{1} - \pd{\tilde{U}}{\bDelta_{2}} \times \bDelta_{2},\,
        -\pd{\tilde{U}}{\delta_{1}} \bDelta_{1} - \pd{\tilde{U}}{\delta_{2}} \bDelta_{2}
      }.
    \end{split}
  \end{equation}

  We would like to stabilize the following equilibrium of the controlled system~\eqref{eq:ControlledEP-uwv-drift} with control~\eqref{eq:u-uwv-drift}:
  \begin{equation*}
    \zeta_{\rm e} \defeq (\bOmega_{\rm e}, \mathbf{v}_{\rm e}, \bGamma_{\rm e}, \bDelta_{\rm 1,e}, \bDelta_{\rm 2,e}, \delta_{\rm 1,e}, \delta_{\rm 2,e})
    = \parentheses{
      \mathbf{0}, \vd, R_{\rm d}^{T}\mathbf{e}_{3}, R_{\rm d}^{T}\mathbf{w}_{1}, R_{\rm d}^{T}\mathbf{w}_{2}, 0, 0
    }.
  \end{equation*}
  Recall from Example~\ref{ex:desired_motion} that $R_{\rm d}$ is the desired orientation and $\vd$ is the desired velocity in the body frame, and also from Example~\ref{ex:drift} that $\{ \mathbf{w}_{1}, \mathbf{w}_{2} \}$ is a basis for the orthogonal complement to $\Span\{ \mathbf{w}_{3} \defeq R_{\rm d} \vd/\norm{ \vd } \}$ in the spatial frame.
  They are defined so that $\{ \mathbf{w}_{1}, \mathbf{w}_{2}, \mathbf{w}_{3} \}$ is a right-handed orthonormal basis, and hence so is $\{ \bDelta_{\rm 1,e}, \bDelta_{\rm 2,e}, \vd/\norm{\vd} \}$.

  The point $\zeta_{\rm e}$ is an equilibrium of the controlled system~\eqref{eq:ControlledEP-uwv-drift} with control~\eqref{eq:u-uwv-drift} if and only if
  \begin{align*}
    \mathbf{u}_{\so(3)^{*}}(\bGamma_{\rm e}, \Delta_{\rm e})
    &= m \mathrm{g} l \boldsymbol{\chi} \times \bGamma_{\rm e} - \bP_{\rm e} \times \mathbf{v}_{\rm d} \\
    &= m \mathrm{g} l \boldsymbol{\chi} \times \bGamma_{\rm e} - \norm{\vd} (M \vd) \times \parentheses{ R_{\rm d}^{T}\mathbf{w}_{3} } \\
    &= m \mathrm{g} l \boldsymbol{\chi} \times \bGamma_{\rm e}- \norm{\vd}(M \vd) \times \parentheses{
      \parentheses{ R_{\rm d}^{T}\mathbf{w}_{1} } \times \parentheses{ R_{\rm d}^{T}\mathbf{w}_{2} }
      } \\
    &= m \mathrm{g} l \boldsymbol{\chi} \times \bGamma_{\rm e} - \norm{\vd} (M \vd) \times (\bDelta_{\rm 1,e} \times \bDelta_{\rm 2,e}) \\
    &= m \mathrm{g} l \boldsymbol{\chi} \times \bGamma_{\rm e} \\
    &\quad - \norm{\vd}\parentheses{ \bDelta_{\rm 2,e} \times (M \vd) } \times \bDelta_{\rm 1,e}
      - \norm{\vd}\parentheses{ (M \vd) \times \bDelta_{\rm 1,e} } \times \bDelta_{\rm 2,e},
    \\
    \mathbf{u}_{(\R^{3})^{*}}(\bGamma_{\rm e}, \Delta_{\rm e})
    &= \mathbf{0},
  \end{align*}
  where we used the shorthand $\Delta_{\rm e} = (\bDelta_{\rm 1,e}, \delta_{\rm 1,e}, \bDelta_{\rm 2,e}, \delta_{\rm 2,e})$.
  On the other hand, the matching condition \eqref{eq:u-uwv-drift} yields
  \begin{align*}
    \mathbf{u}_{\so(3)^{*}}(\bGamma_{\rm e}, \Delta_{\rm e})
    &= -\left.\pd{\tilde{U}}{\bGamma}\right|_{\rm e} \times \bGamma_{\rm e}
      - \left.\pd{\tilde{U}}{\bDelta_{1}}\right|_{\rm e} \times \bDelta_{\rm 1,e}
      - \left.\pd{\tilde{U}}{\bDelta_{2}}\right|_{\rm e} \times \bDelta_{\rm 2,e},
    \\
    \mathbf{u}_{(\R^{3})^{*}}(\bGamma_{\rm e}, \Delta_{\rm e})
    &= -\left.\pd{\tilde{U}}{\delta_{1}}\right|_{\rm e} \bDelta_{\rm 1,e}
      - \left.\pd{\tilde{U}}{\delta_{2}}\right|_{\rm e} \bDelta_{\rm 2,e},
  \end{align*}
  where $(\,\cdot\,)|_{\rm e}$ indicates that the function is evaluated at $(\bGamma_{\rm e}, \Delta_{\rm e})$.
  Therefore, one can achieve matching by requiring $\tilde{U}$ to satisfy
  \begin{gather*}
    \left.\pd{\tilde{U}}{\bGamma}\right|_{\rm e}
    = -m \mathrm{g} l (\boldsymbol{\chi} + \beta \bGamma_{\rm e}),
    \qquad
    \left.\pd{\tilde{U}}{\delta_{1}}\right|_{\rm e} = \left.\pd{\tilde{U}}{\delta_{2}}\right|_{\rm e} = 0,
    \\
    \left.\pd{\tilde{U}}{\bDelta_{1}}\right|_{\rm e}
    = \bDelta_{\rm 2,e} \times (\norm{\vd}(M - \alpha I) \vd),
    \qquad
    \left.\pd{\tilde{U}}{\bDelta_{2}}\right|_{\rm e}
    = (\norm{\vd}(M - \alpha I)\vd) \times \bDelta_{\rm 1,e}
  \end{gather*}
  with arbitrary constants $\alpha, \beta \in \R$; note that $\bDelta_{\rm 2,e} \times \vd = \norm{\vd} \bDelta_{\rm 1,e}$ and $\vd \times \bDelta_{\rm 1,e} = \norm{\vd} \bDelta_{\rm 2,e}$ because $\{ \bDelta_{\rm 1,e}, \bDelta_{\rm 2,e}, \vd/\norm{\vd} \}$ is a right-handed orthonormal basis.
  Using the shorthand
  \begin{equation*}
    \Delta = (\Delta_{1},\Delta_{2}) = (\bDelta_{1}, \delta_{1}, \bDelta_{2}, \delta_{2}),
  \end{equation*}
  a simple form of $\tilde{U}$ satisfying these conditions would be
  \begin{equation*}
    \tilde{U}(\bGamma, \Delta)
    = -m \mathrm{g} l (\boldsymbol{\chi} + \beta \bGamma_{\rm e}) \cdot \bGamma
    + \norm{\vd} \bDelta_{1} \cdot \parentheses{ \bDelta_{2} \times ((M - \alpha I) \vd) }
    + \frac{1}{2} \boldsymbol{\delta}^{T} \mathcal{K} \boldsymbol{\delta}
  \end{equation*}
  with a positive-definite $2 \times 2$ symmetric matrix $\mathcal{K}$; note also that $\boldsymbol{\delta} = (\delta_{1}, \delta_{2}) \in \R^{2}$.
  As a result, we obtain the control
  \begin{align*}
    \mathbf{u}_{\so(3)^{*}}(\bGamma, \Delta)
    &= m \mathrm{g} l (\boldsymbol{\chi} + \beta \bGamma_{\rm e}) \times \bGamma
    + (\bDelta_{1} \times \bDelta_{2}) \times  (\norm{\vd} (M - \alpha I) \vd),
    \\
    \mathbf{u}_{(\R^{3})^{*}}(\bGamma, \Delta)
    &= - [\bDelta_{1} \ \bDelta_{2}] \mathcal{K} \boldsymbol{\delta}.
  \end{align*}

  Note that our formulation uses slightly different variables from those of \cite[Lemma~4.6 and Theorem~4.7]{Le1997b}, and gives a more succinct form of the controlled system---a simpler system with less advected parameters.
  Note also that we obtained (see Example~\ref{ex:LPB-drift} in the Appendix) a simple expression $\bP \cdot (\bDelta_{1} \times \bDelta_{2})$ for the rather awkward-looking Casimir $e_{3}^{T} Q R P$ in \cite{Le1997b}.

  Despite the relative simplicity, our control law turns out to be the same as that of \cite{Le1997b}.
  To see this, first notice that the their expression for $u^{\tau}$ ($\mathbf{u}_{\so(3)^{*}}$ in ours) has $\bTheta \defeq R^{T} Q^{T} \mathbf{e}_{3}$ in place of our $\bDelta_{1} \times \bDelta_{2}$, but then recall from Example~\ref{ex:drift} that $Q \defeq [ \mathbf{w}_{1}\ \mathbf{w}_{2}\ \mathbf{w}_{3} ]^{T}$, and so
  \begin{equation*}
    \bTheta \defeq R^{T} Q^{T} \mathbf{e}_{3}
    = R^{T}[ \mathbf{w}_{1}\ \mathbf{w}_{2}\ \mathbf{w}_{3} ] \mathbf{e}_{3}
    = R^{T}\mathbf{w}_{3} = (R^{T}\mathbf{w}_{1}) \times (R^{T}\mathbf{w}_{2})
    = \bDelta_{1} \times \bDelta_{2},
  \end{equation*}
  hence showing that our $\mathbf{u}_{\so(3)^{*}}$ is the same as their $u^{\tau}$.
  On the other hand, they have $u^{\rm f} = -R^{T} Q^{T} J K \tilde{b}$ with $J =
  \begin{tbmatrix}
    1 & 0 & 0 \\
    0 & 1 & 0 \\
    0 & 0 & 0
  \end{tbmatrix}$ and a positive-definite $3 \times 3$ matrix $K$, but then recall from Example~\ref{ex:drift} their $\tilde{b}$ is related to our $\boldsymbol{\delta}$ as $\tilde{b} = J Q \mathbf{x} = (\boldsymbol{\delta}, 0)$, and so
  \begin{equation*}
    u^{\rm f} = -R^{T} Q^{T} J K J Q \mathbf{x}
    = -\brackets{ R^{T}\mathbf{w}_{1}\ R^{T}\mathbf{w}_{2}\ R^{T}\mathbf{w}_{3} }
    \begin{bmatrix}
      \mathcal{K} & \mathbf{0} \\
      \mathbf{0}^{T} & 0
    \end{bmatrix}
    \begin{bmatrix}
      \boldsymbol{\delta} \\
      0
    \end{bmatrix}
    = -\brackets{ \bDelta_{1}\ \bDelta_{2} }
    \mathcal{K} \boldsymbol{\delta}
  \end{equation*}
  where $\mathcal{K}$ is the upper left $2 \times 2$ submatrix of $K$.
  This is nothing but our $\mathbf{u}_{(\R^{3})^{*}}$.
  Therefore, our control is the same as the one from \cite[Theorem~4.7]{Le1997b}, and hence stabilizes the equilibrium under the conditions given there.
\end{example}

\section{Conclusion}
Advected parameters help us formulate mechanical systems defined on Lie groups with broken symmetry in a simple and effective manner.
One can also keep track of additional parameters of practical interests using proper representations and advected parameters as well.
We focused on those mechanical systems on a semidirect product Lie group $\mathsf{G} \ltimes V$---with a particular focus on $\mathsf{SE}(3) = \mathsf{SO}(3) \ltimes \mathbb{R}^{3}$---with broken symmetry, and derived matching conditions using potential shaping for controlling them.

Specifically, we addressed the following two types of problems:
(i)~applying a control to reduce the advected parameters to obtain a simpler system;
(ii)~tracking and controlling additional advected parameters.
In each of these cases, we found a matching condition for potential shaping.
These matching conditions do not encompass stabilization themselves; instead they must be followed by a stability analysis to ensure stability.

The example for the first setting is a simple ad-hoc potential shaping from our previous work~\cite{CoOh-EPwithBSym1} applied to the heavy top spinning on a movable base.
Although this is a very simple control and does not stabilize the upright spinning position by itself, it is an important first step that facilitates the kinetic shaping to follow to stabilize the equilibrium as shown in \cite{CoOh-EPwithBSym1}.

On the other hand, the second setting provides more versatility.
In fact, our result gives a unified approach to two different problems on controlling underwater vehicles from \cite{Le1997b}, namely stabilization of a desired orientation (Example~\ref{ex:uwv-desired_motion}) and prevention of undesired drift (Example~\ref{ex:uwv-drift}).
Specifically, we have shown that our general matching condition reproduces those controls obtained in \cite{Le1997b} for both settings.
Furthermore, we have demonstrated the utility of our approach---which stresses the role of representations and advected parameters---by showing that it gives a simpler formulation of the problem of preventing undesired drift than that of \cite{Le1997b}.

\section*{Acknowledgments}
We would like to thank the reviewers for their helpful comments.
This work was supported by NSF grant CMMI-1824798.

\appendix

\section{Lie--Poisson Brackets}
While this paper focuses on the Lagrangian formulation of mechanical systems with broken symmetry, one can perform the Legendre transformation to obtain the Hamiltonian formulation of the systems as well.
The main advantage of the Hamiltonian formulation is that it is more useful in finding the Casimirs.

\subsection{Lie--Poisson Bracket on $\mathfrak{s}^{*} = (\mathfrak{g} \ltimes V)^{*}$}
Let $\mathfrak{s} = \mathfrak{g} \ltimes V$ be the Lie algebra of the semidirect product Lie group $\mathsf{S} \defeq \mathsf{G} \ltimes V$.
The $(-)$-Lie--Poisson bracket on $\mathfrak{s}^{*}$ is given by (see \citet{MaRaWe1984a,MaRaWe1984b})
\begin{equation}
  \label{eq:LPB-s}
  \PB{f}{h}_{\mathfrak{s}^{*}}(\mu,p)
  = -\ip{\mu}{ \left[ \fd{f}{\mu}, \fd{h}{\mu} \right] }
    -\ip{p}{ \rho'\parentheses{\fd{f}{\mu}} \pd{h}{p} - \rho'\parentheses{\fd{h}{\mu}} \pd{f}{p} }
\end{equation}
We denote $\mathfrak{s}^{*}$ equipped with $\PB{\,\cdot\,}{\,\cdot\,}_{\mathfrak{s}^{*}}$ by $\mathfrak{s}^{*}$.

\begin{example}[Lie--Poisson bracket on $\se(3)^{*}$]
  If $\mathfrak{g} = \so(3)$ and $V = \R^{3}$, then $\mathfrak{s} = \se(3)$.
  Using the expression for $\rho'$ from \eqref{eq:rho'-se3}, \eqref{eq:LPB-s} yields
  \begin{equation}
    \label{eq:LPB-se3}
    \PB{f}{h}_{\se(3)^{*}}(\bPi,\bP)
    = -\bPi \cdot \parentheses{ \pd{f}{\bPi} \times \pd{h}{\bPi} }
    - \bP \cdot \parentheses{ \pd{f}{\bPi} \times \pd{h}{\bP} - \pd{h}{\bPi} \times \pd{f}{\bP} }.
  \end{equation}
  This is essentially the heavy top bracket upon replacing $\bP$ by $\bGamma$.
  In our context, $\bP$ stands for the linear impulse defined in \eqref{eq:impulses}, and so has a different physical meaning from $\bGamma$.
\end{example}

\subsection{Lie--Poisson Bracket on $(\mathfrak{s} \ltimes X)^{*}$}
We may describe those \textit{uncontrolled} mechanical systems with broken symmetry shown in Section~\ref{ssec:EPwithAdP} as the Lie--Poisson equation on the dual $(\mathfrak{s} \ltimes X)^{*}$ of the semidirect product Lie algebra $\mathfrak{s} \ltimes X$.
Particularly, using the representation $\sigma$ defined in Section~\ref{ssec:recovery}, the Lie--Poisson bracket on $(\mathfrak{s} \ltimes X)^{*}$ is given by
\begin{equation}
  \label{eq:LPB-sX}
  \PB{f}{h}_{(\mathfrak{s} \ltimes X)^{*}}(\mu,p,a)
  = \PB{f}{h}_{\mathfrak{s}^{*}}
  -\ip{a}{ \sigma'\parentheses{\fd{f}{(\mu,p)}} \pd{h}{a} - \sigma'\parentheses{\fd{h}{(\mu,p)}} \pd{f}{a} }.
\end{equation}
Also, by considering a subrepresentation on $(\mathfrak{s} \ltimes X)^{*}$, the \textit{controlled} system~\eqref{eq:EP-tilde-I} with potential shaping using the matching described in Section~\ref{ssec:matching-1} may also be described in terms of the Lie--Poisson bracket on $(\mathfrak{s} \ltimes \tilde{X})^{*}$.

\begin{example}[Lie--Poisson bracket on $(\se(3) \ltimes \R^{3})^{*}$]
  If $\mathfrak{s} = \se(3)$ and $X = \R^{3}$, then, using the bracket \eqref{eq:LPB-se3} and also the expression for $\sigma'$ from \eqref{eq:sigma'-uwv}, \eqref{eq:LPB-sX} gives
  \begin{align*}
    \PB{f}{h}_{(\se(3) \ltimes \R^{3})^{*}}(\bPi,\bP,\bGamma)
    &= -\bPi \cdot \parentheses{ \pd{f}{\bPi} \times \pd{h}{\bPi} }
      - \bP \cdot \parentheses{ \pd{f}{\bPi} \times \pd{h}{\bP} - \pd{h}{\bPi} \times \pd{f}{\bP} } \nonumber\\
    &\quad - \bGamma \cdot \parentheses{ \pd{f}{\bPi} \times \pd{h}{\bGamma} - \pd{h}{\bPi} \times \pd{f}{\bGamma} }.
  \end{align*}
  The uncontrolled underwater vehicle from Example~\ref{ex:uwv} is governed by the Lie--Poisson equation with respect to this bracket.
  Note also that the heavy top on a movable base \textit{after} potential shaping shown in Example~\ref{ex:htmb-matching} is also described in terms of the same bracket.
\end{example}

\subsection{Lie--Poisson Bracket on $(\mathfrak{s} \ltimes (X \times Y))^{*}$}
Matching described in Section~\ref{ssec:matching-2} yields Lie--Poisson equation on the extended $(\mathfrak{s} \ltimes (X \times Y))^{*}$ with the additional parameters living in $Y^{*}$.
Using the representation $\tau$ defined in Section~\ref{ssec:matching-2}, we have the Lie--Poisson bracket on $(\mathfrak{s} \ltimes (X \times Y))^{*}$ as follows:
\begin{align}
  \label{eq:LPB-sXY}
  \PB{f}{h}_{(\mathfrak{s} \ltimes (X \times Y))^{*}}(\mu,p,a,b)
  &= \PB{f}{h}_{\mathfrak{s}^{*}}
    -\ip{a}{ \sigma'\parentheses{\fd{f}{(\mu,p)}} \pd{h}{a} - \sigma'\parentheses{\fd{h}{(\mu,p)}} \pd{f}{a} }
    \nonumber\\
  &\quad - \ip{b}{ \tau'\parentheses{\fd{f}{(\mu,p)}} \pd{h}{b} - \tau'\parentheses{\fd{h}{(\mu,p)}} \pd{f}{b} }.
\end{align}

\begin{example}[Lie--Poisson bracket on $(\se(3) \ltimes (\R^{3} \times (\R^{4} \times \R^{4})))^{*}$]
  \label{ex:LPB-drift}
  Consider the case with $\mathfrak{s} = \se(3)$, $X = \R^{3}$, and $Y = \R^{4} \times \R^{4}$.
  Using the expression for $\tau'$ from \eqref{eq:tau'-drift}, \eqref{eq:LPB-sXY} gives, using the shorthand $\Delta_{i} = (\bDelta_{i},\delta_{i}) \in \R^{4}$ with $i = 1, 2$,
  \begin{align*}
    &\PB{f}{h}_{(\se(3) \ltimes (\R^{3} \times (\R^{4} \times \R^{4})))^{*}}(\bPi, \bP, \bGamma, \Delta_{1}, \Delta_{2}) \nonumber\\
    &\qquad=-\bPi \cdot \parentheses{ \pd{f}{\bPi} \times \pd{h}{\bPi} }
      - \bP \cdot \parentheses{ \pd{f}{\bPi} \times \pd{h}{\bP} - \pd{h}{\bPi} \times \pd{f}{\bP} }
      - \bGamma \cdot \parentheses{ \pd{f}{\bPi} \times \pd{h}{\bGamma} - \pd{h}{\bPi} \times \pd{f}{\bGamma} } \nonumber\\
    &\qquad\quad - \sum_{i=1}^{2} \bDelta_{i} \cdot \parentheses{
      \pd{f}{\bPi} \times \pd{h}{\bDelta_{i}} - \pd{h}{\bPi} \times \pd{f}{\bDelta_{i}}
     - \pd{f}{\delta_{i}} \pd{h}{\bP} + \pd{h}{\delta_{i}} \pd{f}{\bP}
      } \nonumber\\
    &\qquad= \pd{f}{\bPi} \cdot \parentheses{
      \bPi \times  \pd{h}{\bPi} + \bP \times \pd{h}{\bP} + \bGamma \times \pd{h}{\bGamma}
      + \sum_{i=1}^{2} \bDelta_{i} \times \pd{h}{\bDelta_{i}}
      } \nonumber \\
    &\qquad\quad
      + \pd{f}{\bP} \cdot \parentheses{
      \bP \times \pd{h}{\bPi} - \sum_{i=1}^{2} \pd{h}{\delta_{i}} \bDelta_{i}
      }
      + \pd{f}{\bGamma} \cdot \parentheses{ \bGamma \times \pd{h}{\bPi} } \nonumber\\
    &\qquad\quad
      + \sum_{i=1}^{2} \parentheses{
      \pd{f}{\bDelta_{i}} \cdot \parentheses{ \bDelta_{i} \times \pd{h}{\bPi} }
      + \pd{f}{\delta_{i}} \parentheses{ \bDelta_{i} \cdot \pd{h}{\bP} }
      }.
  \end{align*}
  This is the Lie--Poisson bracket for the controlled system~\eqref{eq:ControlledEP-uwv-drift} from Example~\ref{ex:uwv-drift}.
  One sees from the expression that $\bP\cdot(\bDelta_{1} \times \bDelta_{2}), \norm{\bGamma}^{2}$, $\norm{\bDelta_{i}}^{2}$, $\bGamma \cdot \bDelta_{i}$, $\bDelta_{1} \cdot \bDelta_{2}$ with $i = 1,2$ are Casimirs.
\end{example}

\bibliography{CtrlEPwithBSym2}

\begin{thebibliography}{32}
\providecommand{\natexlab}[1]{#1}
\providecommand{\url}[1]{\texttt{#1}}
\expandafter\ifx\csname urlstyle\endcsname\relax
  \providecommand{\doi}[1]{doi: #1}\else
  \providecommand{\doi}{doi: \begingroup \urlstyle{rm}\Url}\fi

\bibitem[Blankenstein et~al.(2002)Blankenstein, Ortega, and van~der
  Schaft]{BlOrVa2002}
G.~Blankenstein, R.~Ortega, and A.~J. van~der Schaft.
\newblock The matching conditions of controlled {L}agrangians and
  {IDA}-passivity based control.
\newblock \emph{International Journal of Control}, 75\penalty0 (9):\penalty0
  645--665, 2002.

\bibitem[Bloch et~al.(1999)Bloch, Leonard, and Marsden]{BlLeMa1999b}
A.~M. Bloch, N.~E. Leonard, and J.~E. Marsden.
\newblock Potential shaping and the method of controlled lagrangians.
\newblock In \emph{Proceedings of the 38th IEEE Conference on Decision and
  Control (Cat. No.99CH36304)}, volume~2, pages 1652--1657 vol.2, 1999.

\bibitem[Bloch et~al.(2001)Bloch, Leonard, and Marsden]{BlLeMa2001}
A.~M. Bloch, N.~E. Leonard, and J.~E. Marsden.
\newblock Controlled {L}agrangians and the stabilization of
  {E}uler--{P}oincar\'e mechanical systems.
\newblock \emph{International Journal of Robust and Nonlinear Control},
  11\penalty0 (3):\penalty0 191--214, 2001.

\bibitem[Bloch et~al.(Dec 2000)Bloch, Leonard, and Marsden]{BlLeMa2000}
A.~M. Bloch, N.~E. Leonard, and J.~E. Marsden.
\newblock Controlled {L}agrangians and the stabilization of mechanical systems.
  {I}. {T}he first matching theorem.
\newblock \emph{IEEE Transactions on Automatic Control}, 45\penalty0
  (12):\penalty0 2253--2270, Dec 2000.

\bibitem[Bloch et~al.(Oct 2001)Bloch, Chang, Leonard, and
  Marsden]{BlChLeMa2001}
A.~M. Bloch, D.~E. Chang, N.~E. Leonard, and J.~E. Marsden.
\newblock Controlled {L}agrangians and the stabilization of mechanical systems.
  {II}. {P}otential shaping.
\newblock \emph{IEEE Transactions on Automatic Control}, 46\penalty0
  (10):\penalty0 1556--1571, Oct 2001.

\bibitem[Borum and Bretl(2014)]{BoBr2014}
A.~D. Borum and T.~Bretl.
\newblock Geometric optimal control for symmetry breaking cost functions.
\newblock \emph{53rd IEEE Conference on Decision and Control}, pages
  5855--5861, 2014.

\bibitem[Borum and Bretl(2016)]{BoBr2016}
A.~D. Borum and T.~Bretl.
\newblock Reduction of sufficient conditions for optimal control problems with
  subgroup symmetry.
\newblock \emph{IEEE Transactions on Automatic Control}, PP\penalty0
  (99):\penalty0 3209--3224, 2016.
\newblock ISSN 0018-9286.

\bibitem[Bullo and Lewis(2004)]{BuLe2004}
F.~Bullo and A.~D. Lewis.
\newblock \emph{Geometric Control of Mechanical Systems}, volume~49 of
  \emph{Texts in Applied Mathematics}.
\newblock Springer, 2004.

\bibitem[Cendra et~al.(1998)Cendra, Holm, Marsden, and Ratiu]{CeHoMaRa1998}
H.~Cendra, D.~D. Holm, J.~E. Marsden, and T.~S. Ratiu.
\newblock Lagrangian reduction, the {E}uler--{P}oincar\'e equations, and
  semidirect products.
\newblock \emph{Amer. Math. Soc. Transl.}, 186:\penalty0 1--25, 1998.

\bibitem[Chang and Marsden(2004)]{ChMa2004}
D.~E. Chang and J.~E. Marsden.
\newblock Reduction of controlled {L}agrangian and {H}amiltonian systems with
  symmetry.
\newblock \emph{SIAM Journal on Control and Optimization}, 43\penalty0
  (1):\penalty0 277--300, 2004.

\bibitem[Chang et~al.(2002)Chang, Bloch, Leonard, Marsden, and
  Woolsey]{ChBlLeMaWo2002}
D.~E. Chang, A.~M. Bloch, N.~E. Leonard, J.~E. Marsden, and C.~A. Woolsey.
\newblock The equivalence of controlled {L}agrangian and controlled
  {H}amiltonian systems.
\newblock \emph{ESAIM: COCV}, 8:\penalty0 393--422, 2002.

\bibitem[Chyba et~al.(2007)Chyba, Haberkorn, Smith, and Wilkens]{ChHaSmWi2007}
M.~Chyba, T.~Haberkorn, R.~N. Smith, and G.~R. Wilkens.
\newblock Controlling a submerged rigid body: A geometric analysis.
\newblock In F.~Bullo and K.~Fujimoto, editors, \emph{Lagrangian and
  Hamiltonian Methods for Nonlinear Control 2006}, pages 375--385, Berlin,
  Heidelberg, 2007. Springer Berlin Heidelberg.

\bibitem[Contreras and Ohsawa()]{CoOh-EPwithBSym1}
C.~Contreras and T.~Ohsawa.
\newblock Controlled {L}agrangians and stabilization of {E}uler--{P}oincar\'e
  mechanical systems with broken symmetry {I}: Kinetic shaping.
\newblock \emph{arXiv:2003.10584}.

\bibitem[Hamberg(1999)]{Ha1999}
J.~Hamberg.
\newblock General matching conditions in the theory of controlled
  {L}agrangians.
\newblock \emph{Decision and Control, 1999. Proceedings of the 38th IEEE
  Conference on}, 3:\penalty0 2519--2523 vol.3, 1999.

\bibitem[Hamberg(2000)]{Ha2000}
J.~Hamberg.
\newblock Controlled {L}agrangians, symmetries and conditions for strong
  matching.
\newblock In \emph{IFAC Lagrangian and Hamiltonian Methods for Nonlinear
  Control}, 2000.

\bibitem[Holm et~al.(1998)Holm, Marsden, and Ratiu]{HoMaRa1998a}
D.~D. Holm, J.~E. Marsden, and T.~S. Ratiu.
\newblock The {E}uler--{P}oincar\'e equations and semidirect products with
  applications to continuum theories.
\newblock \emph{Advances in Mathematics}, 137\penalty0 (1):\penalty0 1--81,
  1998.

\bibitem[Holm et~al.(2009)Holm, Schmah, and Stoica]{HoScSt2009}
D.~D. Holm, T.~Schmah, and C.~Stoica.
\newblock \emph{Geometric mechanics and symmetry: from finite to infinite
  dimensions}.
\newblock Oxford texts in applied and engineering mathematics. Oxford
  University Press, 2009.

\bibitem[Leonard(1997{\natexlab{a}})]{Le1997a}
N.~E. Leonard.
\newblock Stability of a bottom-heavy underwater vehicle.
\newblock \emph{Automatica}, 33\penalty0 (3):\penalty0 331--346,
  1997{\natexlab{a}}.

\bibitem[Leonard(1997{\natexlab{b}})]{Le1997b}
N.~E. Leonard.
\newblock Stabilization of underwater vehicle dynamics with symmetry-breaking
  potentials.
\newblock \emph{Systems \& Control Letters}, 32\penalty0 (1):\penalty0 35--42,
  1997{\natexlab{b}}.

\bibitem[Leonard and Marsden(1997)]{LeMa1997}
N.~E. Leonard and J.~E. Marsden.
\newblock Stability and drift of underwater vehicle dynamics: Mechanical
  systems with rigid motion symmetry.
\newblock \emph{Physica D: Nonlinear Phenomena}, 105\penalty0 (1-3):\penalty0
  130--162, 1997.

\bibitem[Marsden and Ratiu(1999)]{MaRa1999}
J.~E. Marsden and T.~S. Ratiu.
\newblock \emph{Introduction to Mechanics and Symmetry}.
\newblock Springer, 1999.

\bibitem[Marsden et~al.(1984{\natexlab{a}})Marsden, Ratiu, and
  Weinstein]{MaRaWe1984a}
J.~E. Marsden, T.~S. Ratiu, and A.~Weinstein.
\newblock Semidirect products and reduction in mechanics.
\newblock \emph{Transactions of the American Mathematical Society},
  281\penalty0 (1):\penalty0 147--177, 1984{\natexlab{a}}.

\bibitem[Marsden et~al.(1984{\natexlab{b}})Marsden, Ratiu, and
  Weinstein]{MaRaWe1984b}
J.~E. Marsden, T.~S. Ratiu, and A.~Weinstein.
\newblock Reduction and {H}amiltonian structures on duals of semidirect product
  {L}ie algebras.
\newblock In \emph{Fluids and Plasmas : Geometry and Dynamics}, volume~28 of
  \emph{Contemporary Mathematics}. American Mathematical Society,
  1984{\natexlab{b}}.

\bibitem[Nijmeijer and van~der Schaft(1990)]{NiSc1990}
H.~Nijmeijer and A.~van~der Schaft.
\newblock \emph{Nonlinear Dynamical Control Systems}.
\newblock Springer, 1990.

\bibitem[Ortega et~al.(1998)Ortega, Perez, Nicklasson, and
  Sira-Ramirez]{OrPeNiSi1998}
R.~Ortega, J.~Perez, P.~Nicklasson, and H.~Sira-Ramirez.
\newblock \emph{Passivity-based Control of {E}uler--{L}agrange Systems:
  Mechanical, Electrical and Electromechanical Applications}.
\newblock Communications and Control Engineering. Springer London, 1998.

\bibitem[Ortega et~al.(2001)Ortega, {van der Schaft}, Mareels, and
  Maschke]{OrScMaMa2001}
R.~Ortega, A.~J. {van der Schaft}, I.~Mareels, and B.~Maschke.
\newblock Putting energy back in control.
\newblock \emph{IEEE Control Systems}, 21\penalty0 (2):\penalty0 18--33, 2001.

\bibitem[Ortega et~al.(2002)Ortega, Spong, Gomez-Estern, and
  Blankenstein]{OrSpGoBl2002}
R.~Ortega, M.~W. Spong, F.~Gomez-Estern, and G.~Blankenstein.
\newblock Stabilization of a class of underactuated mechanical systems via
  interconnection and damping assignment.
\newblock \emph{IEEE Transactions on Automatic Control}, 47\penalty0
  (8):\penalty0 1218--1233, 2002.

\bibitem[Smith et~al.(2009)Smith, Chyba, Wilkens, and Catone]{SmChWiCa2009}
R.~N. Smith, M.~Chyba, G.~R. Wilkens, and C.~J. Catone.
\newblock A geometrical approach to the motion planning problem for a submerged
  rigid body.
\newblock \emph{International Journal of Control}, 82\penalty0 (9):\penalty0
  1641--1656, 2009.

\bibitem[Spong and Bullo(2005)]{SpBu2005}
M.~W. Spong and F.~Bullo.
\newblock Controlled symmetries and passive walking.
\newblock \emph{IEEE Transactions on Automatic Control}, 50\penalty0
  (7):\penalty0 1025--1031, 2005.

\bibitem[van~der Schaft(1986)]{Sc1986}
A.~J. van~der Schaft.
\newblock Stabilization of {H}amiltonian systems.
\newblock \emph{Nonlinear Analysis: Theory, Methods \& Applications},
  10\penalty0 (10):\penalty0 1021--1035, 1986.

\bibitem[Woolsey and Leonard(2002)]{WoLe2002}
C.~A. Woolsey and N.~E. Leonard.
\newblock Stabilizing underwater vehicle motion using internal rotors.
\newblock \emph{Automatica}, 38\penalty0 (12):\penalty0 2053--2062, 2002.

\bibitem[Woolsey and Techy(2009)]{WoTe2009}
C.~A. Woolsey and L.~Techy.
\newblock Cross-track control of a slender, underactuated {AUV} using potential
  shaping.
\newblock \emph{Ocean Engineering}, 36\penalty0 (1):\penalty0 82--91, 2009.

\end{thebibliography}
\bibliographystyle{plainnat}

\end{document}